\documentclass[11pt, oneside]{article}   	
\usepackage{geometry}                		
\usepackage{graphicx}				
\usepackage{mathrsfs,amssymb,amsmath,amsthm, color,stmaryrd}

\newtheorem{thm}{Theorem}[section]

\newtheorem{defi}[thm]{Definition}
\newtheorem{lem}[thm]{Lemma}
\newtheorem{prop}[thm]{Proposition}

\newtheorem{exam}[thm]{Example}

\newtheorem{rem}[thm]{Remark}
\numberwithin{equation}{section}

\newcommand{\bbN}{\mathbb{N}}

\newcommand{\bbR}{\mathbb{R}}

\newcommand{\mcA}{\mathcal{A}}
\newcommand{\mcB}{\mathcal{B}}
\newcommand{\mcC}{\mathcal{C}}
\newcommand{\mcD}{\mathcal{D}}

\newcommand{\mcI}{\mathcal{I}}
\newcommand{\mcJ}{\mathcal{J}}
\newcommand{\mcK}{\mathcal{K}}
\newcommand{\mcL}{\mathcal{L}}

\newcommand{\mcN}{\mathcal{N}}

\newcommand{\mcR}{\mathcal{R}}

\newcommand{\scM}{\mathscr{M}}
\newcommand{\scT}{\mathscr{T}}

\newcommand{\bfA}{\mathbf{A}}

\newcommand{\bfT}{\mathbf{T}}
\newcommand{\bfG}{\mathbf{G}}

\newcommand{\bfa}{\mathbf{a}}
\newcommand{\bfb}{\mathbf{b}}
\newcommand{\bfc}{\mathbf{c}}

\newcommand{\bfaob}{\mathbf{a}\ominus\mathbf{b}}
\newcommand{\bfcoa}{\mathbf{c}\ominus\mathbf{a}}
\newcommand{\bfcpb}{\mathbf{c}\oplus\beta}

\newcommand{\wco}{w^*}
\newcommand{\vco}{v^*}

\newcommand{\mfs}{\mathfrak{s}}

\newcommand{\id}{\mathop{\text{\rm id}}}

\newcommand{\lp}{\llparenthesis\,}
\newcommand{\rp}{\,\rrparenthesis}

\newcommand{\lb}{\llbracket}
\newcommand{\rb}{\rrbracket}

\newcommand{\tri}{|\!|\!|}

\newcommand{\spa}{\mathop{\text{\rm span}}}

\allowdisplaybreaks

\title{A semigroup approach to the reconstruction theorem and the multilevel Schauder estimate}
\author{Masato Hoshino\footnote{
Graduate School of Engineering Science, Osaka University, 
1-3, Machikaneyama, Toyonaka, Osaka, 560-8531, Japan. 
Email: {\tt hoshino@sigmath.es.osaka-u.ac.jp}
}
}

\date{}							

\begin{document}
\maketitle

\begin{abstract}
The reconstruction theorem and the multilevel Schauder estimate have central roles in the analytic theory of regularity structures by Hairer (2014). Inspired by Otto and Weber's work (2019), we provide elementary proofs for them by using the semigroup of operators. Essentially, we use only the semigroup property and the upper estimates of kernels.
Moreover, we refine the several types of Besov reconstruction theorems considered by Hairer--Labb\'e (2017) and Broux--Lee (2022) and introduce the new framework of ``regularity-integrability structures". The analytic theorems in this paper are applied to the study of quasilinear SPDEs by Bailleul--Hoshino--Kusuoka (2022+) and an inductive proof of the convergence of random models by Bailleul--Hoshino (2023+).
\end{abstract}

\section{Introduction}

In the past decade, the theory of regularity structures \cite{Hai14} has developed as a basic tool for understanding the renormalizations of singular stochastic PDEs. This theory provides a robust framework adopted to a wide class of equations, including the KPZ equation \cite{Hai13, HQ18}, the dynamical $\Phi^4_3$ model \cite{HX18, CMW23}, the dynamical sine-Gordon model \cite{HS16, CHS18}, and so on. 
An important feature of this theory is to express the solution $f$ of the equation as a ``generalized Taylor expansion" of the form
\begin{align}\label{eq:intro1}
f(\cdot)=\sum_\tau a_\tau(x)(\Pi_x\tau)(\cdot)+O(|\cdot-x|^\gamma)
\end{align}
at each point $x$ of the domain, where $\tau$ runs over a finite number of abstract symbols often represented as rooted decorated trees, $(\Pi_x\tau)(\cdot)$ is a given tempered distribution, $a_\tau(x)$ is a coefficient, and $\gamma\in\bbR$ is an order of the error term.
The main difficulty in solving nonlinear SPDEs is how to define the nonlinear functionals of unknown distributions $f$.
To overcome this difficulty, we consider a virtual space spanned by the symbols $\tau$, where the products $\tau\sigma$ are well-defined as long as required, and lift the distribution \eqref{eq:intro1} to the abstract vector field
\begin{align}\label{eq:intro2}
F(x)=\sum_\tau a_\tau(x)\tau
\end{align}
over the domain of $x$. Then the analytic problems for solving SPDEs are split into two steps; (I) show the well-posedness of the equation at the level of vector fields \eqref{eq:intro2}, and (I\!I) after giving a meaning to $\Pi_x\tau$ for all $\tau$, glue the distributions $\Pi_xF(x)$ over all $x$ and reconstruct the global distribution $f$ satisfying \eqref{eq:intro1}. The step (I\!I) is solved by the so-called \emph{reconstruction theorem} \cite[Section 3]{Hai14}. In the step (I), the most important problem is how to lift the convolution with Green function to the operator at the level of vector fields \eqref{eq:intro2}. The \emph{multilevel Schauder estimate} \cite[Section 5]{Hai14} gives a definition and an appropriate estimate for such an operator. These two analytic theorems have central roles in the theory of regularity structures.

The reconstruction theorem and the multilevel Schauder estimate were first proved by Hairer \cite{Hai14}, but the original proofs are quite long. Nowadays, several more elementary proofs are known.
As for the reconstruction theorem, there are the method by Littlewood--Paley theory \cite{GIP15}, the heat semigroup approach \cite{OW19, BH20}, the mollification approach \cite{ST18}, and the convolution method \cite{FH20} inspired by \cite{OW19}.
Without using regularity structures, Caravenna and Zambotti \cite{CZ20} reformulated the reconstruction theorem at the level of \emph{germ}, which is a generalization of the family of distributions $\{\Pi_xF(x)\}$ as above.
Moreover, the original Besov--H\"older ($B_{\infty,\infty}$) type result is extended to the $B_{p,q}$ type Besov setting \cite{HL17, ST18, LPT21, BL22}, Triebel--Lizorkin setting \cite{HR20}, the quasinormed setting \cite{ZK23}, and Riemannian manifolds \cite{DDD19, RS21, HS23+}.
As for the multilevel Schauder estimate, there is an alternative proof by the heat semigroup approach \cite{BH20}. Also, the original estimate is extended to Besov setting \cite{HL17}, Riemannian manifolds \cite{DDD19, HS23+}, and the germ setting \cite{BCZ23}.

The main purpose of this paper is to formulate the semigroup approach as in \cite{OW19, BH20} in a more general setting and to provide short proofs of the reconstruction theorem (Theorem \ref{thm:besovreconstruction} below) and the multilevel Schauder estimate (Theorem \ref{thm:besovschauder} below).
To shorten the proof, we introduce the Besov space associated with the semigroup of operators and reformulate the theorems in more suitable senses.
In this approach, we essentially need only the semigroup property of operators and upper heat kernel estimates (see Definitions \ref{asmp1} and \ref{asmp2} below), and the proofs are elementary and very short.
Another remarkable point is that we allow \emph{inhomogeneous} operators; the integral kernel $Q_t(x,y)$ is not necessarily to be a function of $x-y$. Such an extension is required in the study of quasilinear SPDEs \cite{BHK22}.
Moreover, in the author's knowledge, the semigroup approach has not been tried for the proof of the multilevel Schauder estimate, except at \cite{BH20}. 

Another purpose of this paper is to introduce the new framework which would be more suitable for the problems involving both regularity and integrability exponents, for example, problems involving Malliavin calculus. Since Cameron--Martin space of Wiener space is typically an $L^2$-Sobolev space, if we consider two different symbols $\Xi$ and $\dot{\Xi}$ representing elements of Wiener space and Cameron--Martin space respectively, it would be convenient to give each symbol the different integrability exponents ``$\infty$" and ``$2$" respectively. To describe such a situation, we introduce the ``regularity-integrability structures" in Section \ref{section:RI} and extend the analytic theorems to this new structure.
In the author's knowledge, such an extension is not known in the literature. Indeed, in the papers \cite{HL17, ST18, LPT21}, the authors considered only $B_{\infty,\infty}$ type models and $B_{p,q}$ type modelled distributions.
On the other hand, our situation seems to be a special case of the germ setting \cite{BL22}, but in the paper \cite{BPHZ}, more detailed structure on the model space is effectively used for an inductive proof of the convergence of random models.

This paper is organized as follows.
In Section \ref{section:semigroup}, we define the Besov spaces associated with the semigroup of operators.
In Section \ref{section:RI}, we introduce the regularity-integrability structures and extend the definitions of models and modelled distributions.
The main parts of this paper are Sections \ref{section:reconstruction} and \ref{section:schauder}, which are devoted to the proofs of the reconstruction theorem and the multilevel Schauder estimate respectively.

\section{Besov spaces associated with the semigroup of operators}\label{section:semigroup}

In this section, we define the Besov norms associated with the semigroup of operators. 
For the sake of generality, we define the weighted Besov norms with arbitrary integrability exponents $p,q\in[1,\infty]$.

\subsection{Notations}\label{subsection:notation}

The symbol $\bbN$ denotes the set of all nonnegative integers.
Throughout this paper, we fix an integer $d\ge1$, the \emph{scaling} $\mfs=(\mfs_1,\dots,\mfs_d)\in[1,\infty)^d$, and a number $\ell>0$. We define $|\mfs|=\sum_{i=1}^d\mfs_i$.
For any multiindex ${\bf k}=(k_i)_{i=1}^d\in\bbN^d$, any $x=(x_i)_{i=1}^d\in\bbR^d$, and any $t>0$, we define the notations
\begin{align*}
&{\bf k}!:=\prod_{i=1}^dk_i!,\quad
|{\bf k}|_{\mfs}:=\sum_{i=1}^d\mfs_ik_i,\quad
\|x\|_\mfs:=\sum_{i=1}^d|x_i|^{1/\mfs_i},\\
&x^{\bf k}:=\prod_{i=1}^dx_i^{k_i},\quad
t^{\mfs/\ell}x:=(t^{\mfs_i/\ell}x_i)_{i=1}^d,\quad
t^{-\mfs/\ell}x:=(t^{-\mfs_i/\ell}x_i)_{i=1}^d.
\end{align*}
We define the set $\bbN[\mfs]:=\{|{\bf k}|_\mfs\,;\,{\bf k}\in\bbN^d\}$, which will be used in Section \ref{section:schauder}. The parameter $t$ is not a physical time variable, but an auxiliary variable used to define regularities of distributions.
For multiindices ${\bf k}=(k_i)_{i=1}^d$ and ${\bf l}=(l_i)_{i=1}^d$, we write ${\bf l}\le{\bf k}$ if $l_i\le k_i$ for any $1\le i\le d$, and then define $\binom{{\bf k}}{{\bf l}}:=\prod_{i=1}^d\binom{k_i}{l_i}$.

We also fix a nonnegative measurable function $G:\bbR^d\to\bbR$
and define for any $t>0$,
$$
G_t(x)=t^{-|\mfs|/\ell}G\big(t^{-\mfs/\ell}x\big).
$$

We use the notation $A\lesssim B$ for two functions $A(x)$ and $B(x)$ of a variable $x$, if there exists a constant $c>0$ independent of $x$ such that $A(x)\le cB(x)$ for any $x$.

\subsection{Weighted $L^p$ spaces}

First we introduce the class of appropriate weight functions.

\begin{defi}\label{defweight}
A continuous function $w:\bbR^d\to(0,1]$ is called a \emph{weight}.
A weight $w$ is said to be \emph{$G$-controlled} if there exists a continuous function $\wco:\bbR^d\to[1,\infty)$ such that
\begin{align}\label{weightmoderate}
w(x+y)\le \wco(x)w(y)
\end{align}
for any $x,y\in\bbR^d$ and
\begin{align}\label{weightintegrable}
\sup_{0<t\le T}\sup_{x\in\bbR^d}\Big\{\|x\|_\mfs^n\,\wco\big(t^{\mfs/\ell}x\big)G(x)\Big\}<\infty
\end{align}
for any $n\ge0$ and $T>0$. 
(In the terminology of \cite[Definition 2.3]{MW17}, $w$ is said to be \emph{$\wco$-moderate}.)
\end{defi}

\begin{defi}
For any $p\in[1,\infty]$ and any weight $w$, we define the weighted $L^p$ norm of a measurable function $f:\bbR^d\to\bbR$ by
\begin{align*}
\|f\|_{L^p(w)}:=\|fw\|_{L^p(\bbR^d)}.
\end{align*}
We denote by $L^p(w)$ the space of all measurable functions with finite $L^p(w)$ norms, and define the closed subspace $L_c^p(w)$ as the completion of the set $C(\bbR^d)\cap L^p(w)$ under the $L^p(w)$ norm.
\end{defi}

Note that $\|\cdot\|_{L^p(w)}$ is nondegenerate because $w$ is fully supported in $\bbR^d$ by definition.
An advantage of introducing $L_c^p(w)$ is that we can use density arguments. Although the space $L_c^\infty(w)$ is strictly smaller than $L^\infty(w)$, it is often sufficient to consider $L_c^\infty(w)$ in applications to SPDEs.
In the following, we prove three useful inequalities.

\begin{lem}\label{lem:weightedholder}
Let $p,q,r\in[1,\infty]$ be such that $1/r=1/p+1/q$.
For any weights $w_1,w_2$ and functions $f\in L^p(w_1)$ and $g\in L^q(w_2)$, we have
$$
\|fg\|_{L^r(w_1w_2)}\le\|f\|_{L^p(w_1)}\|g\|_{L^q(w_2)}.
$$
\end{lem}

\begin{proof}
Since $\|fg\|_{L^r(w_1w_2)}=\|(fw_1)(gw_2)\|_{L^r}$, the result follows from H\"older's inequality.
\end{proof}

\begin{lem}[{\cite[Theorem 2.4]{MW17}}]\label{lem:weightedyoung}
Let $w$ be a $G$-controlled weight.
For any $T>0$, there exists a constant $C_T>0$ depending only on $G,\wco$, and $T$, such that,
for any $t\in(0,T]$, $1\le p\le q\le\infty$, and $f\in L^p(w)$, we have
\begin{align*}
\|G_t*f\|_{L^q(w)}\le C_T\,t^{-\frac{|\mfs|}\ell(\frac1p-\frac1q)}\|f\|_{L^p(w)}.
\end{align*}
\end{lem}

\begin{proof}
Since $|(G_t*f)(x)w(x)|\le\big((G_t\wco)*(|f|w)\big)(x)$
by the inequality \eqref{weightmoderate}, the result follows from Young's inequality. The proportional constant is $\|G_t\wco\|_{L^r(\bbR^d)}$, where $\frac1r=1+\frac1q-\frac1p$. Since
\begin{align*}
\|G_t\wco\|_{L^r(\bbR^d)}&=\big\|t^{-|\mfs|/\ell}G\big(t^{-\mfs/\ell}x\big)\wco(x)\big\|_{L_x^r(\bbR^d)}\\
&=t^{-\frac{|\mfs|}\ell(1-\frac1r)}\big\|G(x)\wco\big(t^{\mfs/\ell}x\big)\big\|_{L_x^r(\bbR^d)}
=t^{-\frac{|\mfs|}\ell(\frac1p-\frac1q)}\big\|G(x)\wco\big(t^{\mfs/\ell}x\big)\big\|_{L_x^r(\bbR^d)}
\end{align*}
by the scaling property, we have the result 
by using the condition \eqref{weightintegrable}.
\end{proof}

\begin{lem}\label{lem:weightedtranslation}
Let $w$ be a $G$-controlled weight.
For any $p\in[1,\infty]$, $f\in L^p(w)$, and $h\in\bbR^d$, we have
\begin{align*}
\|f(\cdot-h)\|_{L^p(w)}\le \wco(h)\|f\|_{L^p(w)}.
\end{align*}
\end{lem}

\begin{proof}
The result follows from the inequality $|f(x-h)|w(x)\le \wco(h)|f(x-h)|w(x-h)$ and the translation invariance of the unweighted $L^p(\bbR^d)$ norm.
\end{proof}

\subsection{Semigroup of operators}

We introduce a semigroup of operators.

\begin{defi}\label{asmp1}
We call a family of continuous functions $\{Q_t:\bbR^d\times\bbR^d\to\bbR\}_{t>0}$ a \emph{$G$-type semigroup (of operators)} if it satisfies the following properties.
\begin{enumerate}
\renewcommand{\theenumi}{(\roman{enumi})}
\renewcommand{\labelenumi}{(\roman{enumi})}
\item\label{asmp1:semigroup}
(Semigroup property)
For any $0<s<t$ and $x,y\in\bbR^d$,
$$
\int_{\bbR^d}Q_{t-s}(x,z)Q_s(z,y)dz=Q_t(x,y).
$$
\item\label{asmp1:conservative}
(Conservativity)
For any $x\in\bbR^d$,
$$
\lim_{t\downarrow0}\int_{\bbR^d}Q_t(x,y)dy=1.
$$
\item\label{asmp1:gauss}
(Upper $G$-type estimate)
There exists a constant $C_1>0$ such that, for any $t>0$ and $x,y\in\bbR^d$,
$$
|Q_t(x,y)|\le C_1G_t(x-y).
$$
\item\label{asmp2:derivative}
(Time derivative)
For any $x,y\in\bbR^d$, $Q_t(x,y)$ is differentiable with respect to $t$. Moreover, there exists a constant $C_2>0$ such that, for any $t>0$ and $x,y\in\bbR^d$,
$$
|\partial_tQ_t(x,y)|\le C_2\, t^{-1}G_t(x-y).
$$
\end{enumerate}
\end{defi}

\begin{exam}\label{exam:FS}
We have in mind a fundamental solution of an anisotropic parabolic operator with bounded and H\"older continuous coefficients
\begin{align}\label{exam:parabolicop}
\partial_t-P(x,\partial_x):=\partial_t-\sum_{|{\bf k}|_\mfs\le \ell}a_{\bf k}(x)\partial_x^{\bf k},
\end{align}
where we suppose that $\ell>\max_{1\le i\le d}\mfs_i$ and $P$ satisfies 
the uniform ellipticity
$$
\mathop{\text{\rm Re}}P(x,i\xi):=
\mathop{\text{\rm Re}}\sum_{|{\bf k}|_\mfs\le \ell}a_{\bf k}(x)(i\xi)^{\bf k}\le-C\|\xi\|_\mfs^\ell,\qquad\xi\in\bbR^d
$$
for some constant $C>0$. As shown in \cite[Appendix A]{BHK22}, the unique solution $Q_t(x,y)$ of
$$
\left\{
\begin{aligned}
\big(\partial_t-P(x,\partial_x)\big)Q_t(x,y)&=0,&&t>0,\ x,y\in\bbR^d,\\
\lim_{t\downarrow0}Q_t(x,\cdot)&=\delta(x-\cdot),\quad&&x\in\bbR^d,
\end{aligned}
\right.
$$
(where $\delta$ is Dirac's delta and the latter convergence is in the distributional sense) satisfies the properties in Definition \ref{asmp1}
with the function
\begin{align}\label{eq:Gchoice}
G(x)=\exp\bigg\{-c\sum_{i=1}^d|x_i|^{\ell/(\ell-\mfs_i)}\bigg\}
\end{align}
for some $c>0$.
An elementary example is the isotropic operator $P(\partial_x)=\Delta$, where $\mfs=(1,1,\dots,1)$ and $\ell=2$. Then $Q_t(x,y)$ is the usual heat kernel $\frac1{(4\pi t)^{d/2}}\exp(-\frac{|x-y|^2}{4t})$ and we can choose $G(x)=e^{-c|x|^2}$, where $|\cdot|$ is the Euclidean norm.
Another example considered in \cite{OW19, BH20} is given by
$$
P(\partial_x):=\partial_{x_1}^2-\Delta_{x'}^2,\qquad\Delta_{x'}:=\sum_{i=2}^d\partial_{x_i}^2,
$$
which is more suitable for parabolic problems. Here, $x_1$ and $x':=(x_i)_{i=2}^d$ are considered as temporal and spatial variables, respectively. In this case, we can choose $\mfs=(2,1,1,\dots,1)$, $\ell=4$, and
$$
G(x)=\exp\big\{-c\big(|x_1|^2+|x'|^{4/3}\big)\big\}.
$$

We return to the general case \eqref{exam:parabolicop}.
As for $G$-controlled weights, the most trivial choice is the flat function $w=1$.
Another weight we consider in \cite{BPHZ} is the function
$$
w(x)=e^{-a\|x\|_{\mfs}}
$$
with some $a>0$.
We can easily see that $\wco(x)=e^{a\|x\|_{\mfs}}$ satisfies \eqref{weightmoderate} by using the triangle inequality of $\|\cdot\|_\mfs$.
The condition \eqref{weightintegrable} holds because the variable inside the exponential function of \eqref{eq:Gchoice} is superlinear with respect to $\|x\|_\mfs$.
\end{exam}

We identify the function $Q_t(x,y)$ with a continuous linear operator on $L^p(w)$ by
$$
(Q_tf)(x):=:Q_t(x,f):=\int_{\bbR^d}Q_t(x,y)f(y)dy,\qquad f\in L^p(w),\ x\in\bbR^d.
$$
Note that $Q_t$ is closed in $L^p(w)$ because of Definition \ref{asmp1}-\ref{asmp1:gauss} and Lemma \ref{lem:weightedyoung}.

\begin{prop}\label{prop:operatorQt}
Let $w$ be a $G$-controlled weight and let $p\in[1,\infty]$.
For any $f\in L^p(w)$ and $t>0$, $Q_tf$ is a continuous function.
In addition, if $f\in C(\bbR^d)\cap L^p(w)$, then
$$
\lim_{t\downarrow0}(Q_tf)(x)=f(x)
$$
for any $x\in\bbR^d$.
\end{prop}

\begin{proof}
Let $f\in L^p(w)$. To show the continuity of $(Q_tf)(x)$ with respect to $x$, it is sufficient to consider the case $t=1$ and $x=0$. Note that, in the region $\|x\|_\mfs\le1$, we have
\begin{align*}
|Q_1(x,y)f(y)w(x)|
&\lesssim |G(x-y)\wco(x-y)||f(y)w(y)|\lesssim \frac{|f(y)w(y)|}{1+\|y\|_\mfs^n}
\end{align*}
for any $n\ge0$ by the property \eqref{weightintegrable}. This implies that
\begin{align*}
\lim_{x\to0}(Q_1f)(x)w(x)=\int_{\bbR^d}\lim_{x\to0}Q_1(x,y)f(y)w(x)dy=(Q_1f)(0)w(0)
\end{align*}
by Lebesgue's convergence theorem. 
Since $w$ is strictly positive and continuous, we have $\lim_{x\to0}(Q_1f)(x)=(Q_1f)(0)$.

Next let $f\in C(\bbR^d)\cap L^p(w)$. To show the continuity with respect to $t$, it is sufficient to consider the case $x=0$. For any $\varepsilon>0$, we can choose $\delta>0$ such that $|f(y)-f(0)|<\varepsilon$ for any $\|y\|_\mfs<\delta$, and have
\begin{align*}
|w(0)&(Q_tf-f)(0)|=w(0)\bigg|\int_{\bbR^d}Q_t(0,y)\big(f(y)-f(0)\big)dy+\bigg(\int_{\bbR^d}Q_t(0,y)dy-1\bigg)f(0)\bigg|\\
&\le w(0)\varepsilon\int_{\|y\|_\mfs<\delta}G_t(-y)dy+w(0)\int_{\|y\|_\mfs\ge\delta}G_t(-y)|f(y)|dy\\
&\quad+w(0)\int_{\|y\|_\mfs\ge\delta}G_t(-y)|f(0)|dy+w(0)|f(0)|\bigg|\int_{\bbR^d}Q_t(0,y)dy-1\bigg|.
\end{align*}
In the far right-hand side, the only nontrivial part is the second term. We bound it from above by
$$
\int_{\|y\|_\mfs\ge\delta}G_t(-y)\wco(-y)|f(y)|w(y)dy\le\|(G_t\wco)(y)\|_{L^{p'}(\|y\|_\mfs\ge\delta)}\|fw\|_{L^p(\bbR^d)},
$$
where $1/p+1/{p'}=1$. We then have that $\|(G_t\wco)(y)\|_{L^{p'}(\|y\|_\mfs\ge\delta)}\to0$ as $t\downarrow0$ by the property \eqref{weightintegrable}.
\end{proof}

\subsection{Besov spaces associated with semigroup}

In what follows, we fix a $G$-type semigroup $\{Q_t\}_{t>0}$ and a $G$-controlled weight $w$.
We define the weighted Besov spaces associated with $\{Q_t\}_{t>0}$, as studied in \cite{BDY12, GL15, BB16}.

\begin{defi}
For every $\alpha\le0$ and $p,q\in[1,\infty]$, we define the Besov space $B_{p,q}^{\alpha,Q}(w)$ as the completion of $L_c^p(w)$ under the norm
$$
\|f\|_{B_{p,q}^{\alpha,Q}(w)}:=\|Q_1f\|_{L^p(w)}+\big\|t^{-\alpha/\ell}\|Q_t f\|_{L^p(w)}\big\|_{L^q(0,1;t^{-1}dt)}.
$$
\end{defi}

When $\mfs=(1,1,\dots,1)$, $\ell=2$, and $Q_t$ is the heat semigroup $e^{t\Delta}$, the above norm (with $\alpha<0$ and $w=1$) is equivalent to the classical Besov norm in Euclidean setting. See e.g., \cite[Theorem 2.34]{BCD11} or \cite[Theorem 2.6.4]{Tri92}.

\begin{rem}\label{QtsymmetricS'}
We can see that $\|\cdot\|_{B_{p,q}^{\alpha,Q}(w)}$ is nondegenerate in $L_c^p(w)$ by the temporal continuity of $Q_t$ (Proposition \ref{prop:operatorQt}) and the density argument. This is the only reason why we define the Besov spaces from $L_c^p(w)$ as above.
On the other hand, if $Q_t$ is symmetric in the sense that $Q_t(y,x)=Q_t(x,y)$ for any $x,y\in\bbR^d$, then for any locally integrable function $f$ and $\varphi\in C_0^\infty(\bbR^d)$, we have
$$
\int_{\bbR^d}(Q_tf)(x)\varphi(x)dx=\int_{\bbR^d}f(x)(Q_t\varphi)(x)dx\xrightarrow{t\downarrow0} \int_{\bbR^d}f(x)\varphi(x)dx,
$$
which implies $Q_tf\to f$ as $t\downarrow0$ in the distributional sense. In this case, $\|\cdot\|_{B_{p,q}^{\alpha,Q}(w)}$ is nondegenerate in whole $L^p(w)$.
\end{rem}

The following result implies that we can ignore the difference of the parameter $q$ at the cost of infinitesimal difference of the parameter $\alpha$.
Therefore, we pay less attention to $q$ in this paper and write $B_p^{\alpha,Q}(w):=B_{p,\infty}^{\alpha,Q}(w)$.

\begin{prop}\label{prop:almostequiv}
For any $f\in L^p(w)$, the following inequalities hold.
\begin{enumerate}
\renewcommand{\theenumi}{(\arabic{enumi})}
\renewcommand{\labelenumi}{(\arabic{enumi})}
\item\label{almostequiv2}
For any $1\le q_1\le q_2\le\infty$ and $\alpha_1<\alpha_2\le0$, we have $\|f\|_{B_{p,q_1}^{\alpha_1,Q}(w)}\lesssim\|f\|_{B_{p,q_2}^{\alpha_2,Q}(w)}$.
\item\label{almostequiv3}
For any $\alpha\le0$, we have $\|f\|_{B_{p,\infty}^{\alpha,Q}(w)}\lesssim\|f\|_{B_{p,1}^{\alpha,Q}(w)}$.
\end{enumerate}
Here the implicit proportional constants depend only on $G,\wco$, and the regularity and integrability exponents.
\end{prop}

\begin{rem}
As a result of \ref{almostequiv2} and \ref{almostequiv3}, we have
\begin{align*}
\|f\|_{B_{p,q_1}^{\alpha-2\varepsilon,Q}(w)}
\lesssim\|f\|_{B_{p,\infty}^{\alpha-\varepsilon,Q}(w)}
\lesssim\|f\|_{B_{p,1}^{\alpha-\varepsilon,Q}(w)}
\lesssim\|f\|_{B_{p,q_2}^{\alpha,Q}(w)}
\end{align*}
for any $\alpha\le0$, $\varepsilon>0$, and $q_1,q_2\in[1,\infty]$.
\end{rem}

\begin{proof}
For \ref{almostequiv2}, taking $r\in[1,\infty]$ such that $1/q_1=1/r+1/q_2$, we have
\begin{align*}
&\big\|t^{-\alpha_1/\ell}\|Q_t f\|_{L^p(w)}\big\|_{L^{q_1}(0,1;t^{-1}dt)}\\
&\le\big\|t^{(\alpha_2-\alpha_1)/\ell}\big\|_{L^r(0,1;t^{-1}dt)}\big\|t^{-\alpha_2/\ell}\|Q_t f\|_{L^p(w)}\big\|_{L^{q_2}(0,1;t^{-1}dt)}
\end{align*}
by H\"older's inequality.
Since $\big\|t^{(\alpha_2-\alpha_1)/\ell}\big\|_{L^r(0,1;t^{-1}dt)}<\infty$, we have the result. 
Next we prove \ref{almostequiv3}. By using Definition \ref{asmp1}-\ref{asmp2:derivative} and Lemma \ref{lem:weightedyoung}, we have
\begin{align*}
\|Q_tf-Q_1f\|_{L^p(w)}
&\le\int_t^1\|\partial_sQ_sf\|_{L^p(w)}ds
=\int_t^1\|(\partial_sQ)_{s/2}Q_{s/2}f\|_{L^p(w)}ds\\
&\lesssim \int_t^1\|Q_{s/2}f\|_{L^p(w)}\frac{ds}s
=\int_{t/2}^{1/2}\|Q_sf\|_{L^p(w)}\frac{ds}s.
\end{align*}
Therefore,
\begin{align*}
t^{-\alpha/\ell}\|Q_tf\|_{L^p(w)}
&\lesssim t^{-\alpha/\ell}\bigg(\|Q_1f\|_{L^p(w)}+\int_{t/2}^1\|Q_sf\|_{L^p(w)}\frac{ds}s\bigg)\\
&\lesssim\|Q_1f\|_{L^p(w)}+\int_{t/2}^1s^{-\alpha/\ell}\|Q_sf\|_{L^p(w)}\frac{ds}s
\le\|f\|_{B_{p,1}^{\alpha,Q}(w)}.
\end{align*}
By taking the supremum over $t\in(0,1]$, we have the result.
\end{proof}

The following result is an analogue of the classical Besov embedding.

\begin{prop}
Let $\alpha\le0$, $p,q,r\in[1,\infty]$, and $r\ge p$. For any $f\in L^p(w)$, we have the inequality
$$
\|f\|_{B_{r,q}^{\alpha-|\mfs|(\frac1p-\frac1r),Q}(w)}\lesssim\|f\|_{B_{p,q}^{\alpha,Q}(w)}.
$$
\end{prop}

\begin{proof}
Since $\|Q_tf\|_{L^r(w)}=\|Q_{t/2}(Q_{t/2}f)\|_{L^r(w)}\lesssim t^{-\frac{|\mfs|}\ell(\frac1p-\frac1r)}\|Q_{t/2}f\|_{L^p(w)}$ by Lemma \ref{lem:weightedyoung}, the result follows from the definition of norms.
\end{proof}

We have the hierarchy between Besov spaces with different parameters $\alpha$.

\begin{prop}\label{prop:abinjection}
Let $\alpha_1<\alpha_2\le0$ and $p\in[1,\infty]$.
The identity $\iota_{\alpha_1}:L_c^p(w)\hookrightarrow B_p^{\alpha_1,Q}(w)$ is uniquely extended to the continuous injection $\iota_{\alpha_1}^{\alpha_2}:B_p^{\alpha_2,Q}(w)\hookrightarrow B_p^{\alpha_1,Q}(w)$.
\end{prop}

\begin{proof}
We prove only the injectivity.
Note that, for any $\alpha\le0$, the operator $Q_t:L_c^p(w)\to L_c^p(w)$ is continuously extended to the operator $Q_t^\alpha:B_p^{\alpha,Q}(w)\to L_c^p(w)$ and it holds that
\begin{align}\label{2:eq:besovnormex}
\|f\|_{B_p^{\alpha,Q}(w)}=\|Q_1^\alpha f\|_{L^p(w)}+\sup_{0<t\le1}t^{-\alpha/\ell}\|Q_t^\alpha f\|_{L^p(w)}
\end{align}
for any $f\in B_p^{\alpha,Q}(w)$.
Let $f\in B_p^{\alpha_2,Q}(w)$ be such that $\iota_{\alpha_1}^{\alpha_2} f=0$ in $B_p^{\alpha_1,Q}(w)$.
Taking a sequence $\{f_n\}\subset L_c^p(w)$ such that $f_n\to f$ in $B_p^{\alpha_2,Q}(w)$, we have $f_n=\iota_{\alpha_1}^{\alpha_2}f_n\to\iota_{\alpha_1}^{\alpha_2}f$ in $B_p^{\alpha_1,Q}(w)$ by the continuity.
By the continuity of $Q_t^{\alpha_i}$ ($i\in\{1,2\}$), we have
$$
Q_t^{\alpha_2}f=\lim_{n\to\infty}Q_tf_n=Q_t^{\alpha_1}(\iota_{\alpha_1}^{\alpha_2}f)=0
$$
in $L^p(w)$ for any $t\in(0,1]$. By the identity \eqref{2:eq:besovnormex}, we have $\|f\|_{B_p^{\alpha_2,Q}(w)}=0$.
\end{proof}

The extensions $\{Q_t^\alpha\}_{0<t\le 1}$ obtained in the above proof are compatible in the sense that $Q_t^{\alpha_1}\circ\iota_{\alpha_1}^{\alpha_2}=Q_t^{\alpha_2}$. Because of this, we can omit the letter $\alpha$ and use the notation $Q_t$ to mean its extension $Q_t^\alpha$ regardless of its domain.
We close this section with the continuity of $Q_t$ with respect to $t$ in Besov norms.

\begin{lem}\label{lem:timecontinuity}
Let $\alpha\le0$ and $p\in[1,\infty]$.
There exists a constant $C>0$ such that, for any $f\in B_p^{\alpha,Q}(w)$, $t\in(0,1]$, and $\varepsilon\in[0,\ell]$, we have
$$
\|(Q_t-\id)f\|_{B_p^{\alpha-\varepsilon,Q}(w)}
\le C\, t^{\varepsilon/\ell}\|f\|_{B_p^{\alpha,Q}(w)}.
$$
\end{lem}

\begin{proof}
Similarly to the proof of Proposition \ref{prop:almostequiv}, we have for any $s,t\in(0,1]$,
\begin{align*}
\|Q_s(Q_t-\id)f\|_{L^p(w)}
&=\|(Q_{t+s}-Q_s)f\|_{L^p(w)}\le\int_s^{t+s}\|\partial_rQ_rf\|_{L^p(w)}dr\\
&\lesssim\int_s^{t+s}\|Q_{r/2}f\|_{L^p(w)}\frac{dr}r
\le\|f\|_{B_p^{\alpha,Q}(w)}\int_s^{t+s}r^{\alpha/\ell-1}dr.
\end{align*}
Since $\int_s^{t+s}r^{\alpha/\ell-1}dr\lesssim (ts^{\alpha/\ell-1})\wedge s^{\alpha/\ell}$, we have the result by an interpolation.
\end{proof}

\section{Basic notions of regularity-integrability structures}\label{section:RI}

In this section, we extend the original definitions of regularity structures, models, and modelled distributions in \cite{Hai14} by taking integrability exponents into account.

\subsection{Regularity-integrability structures}

While the label set of the regularity structure is a set of real numbers, our label set is a subset of $\bbR\times[1,\infty]$.
We denote generic elements of $\bbR\times[1,\infty]$ by bold letters $\bfa,\bfb,\bfc$, and so on. For each element $\bfa=(\alpha,p)\in\bbR\times[1,\infty]$, we write $\alpha=r(\bfa)$ and $p=i(\bfa)$, where the letters ``$r$" and ``$i$" mean ``regularity" and ``integrability", respectively.
We define a partial order $\preceq$ and a strict partial order $\prec$ of the set $\bbR\times[1,\infty]$ by
\begin{align*}
\bfb\preceq\bfa&\quad\overset{\text{def}}{\Leftrightarrow}\quad r(\bfb)\le r(\bfa),\ i(\bfb)\ge i(\bfa),\\
\bfb\prec\bfa&\quad\overset{\text{def}}{\Leftrightarrow}\quad r(\bfb)<r(\bfa),\ i(\bfb)\ge i(\bfa).
\end{align*}
Note that $i(\bfb)$ may be equal to $i(\bfa)$ even for the latter case. For any $\bfb\preceq\bfa$, we define the element $\bfaob\in\bbR\times[1,\infty]$ by
$$
\bfaob:=\Bigg(r(\bfa)-r(\bfb),\frac1{\frac1{i(\bfa)}-\frac1{i(\bfb)}}\Bigg),
$$
where $1/\infty:=0$ and $1/0:=\infty$. In what follows, the relations $r(\bfa)=r(\bfa\ominus\bfb)+r(\bfb)$ and $1/{i(\bfa)}=1/{i(\bfaob)}+1/{i(\bfb)}$ are important.

\begin{defi}
A \emph{regularity-integrability structure} $\scT=(\bfA,\bfT,\bfG)$ consists of the following objects.
\begin{itemize}
\item[(1)]
(Index set) $\bfA$ is a subset of $\bbR\times[1,\infty]$ such that, for every $\bfa\in\bbR\times[1,\infty]$, the subset $\{\bfb\in \bfA\, ;\, \bfb\prec\bfa\}$ is finite.
\item[(2)]
(Model space) $\bfT=\bigoplus_{{\bf a}\in \bfA}\bfT_{\bf a}$ is an algebraic sum of Banach spaces $(\bfT_{\bf a},\|\cdot\|_{\bf a})$.
\item[(3)]
(Structure group) $\bfG$ is a group of continuous linear operators on $\bfT$ such that, for any $\Gamma\in \bfG$ and $\bfa\in\bfA$,
$$
(\Gamma-\id)\bfT_{\bf a}\subset\bigoplus_{{\bf b}\in \bfA,\, {\bf b}\prec {\bf a}}\bfT_{\bf b}.
$$
\end{itemize}
A \emph{regularity} of $\scT$ is $\alpha_0\in\bbR$ such that $(\alpha_0,\infty)\preceq\bfa$ for any $\bfa\in \bfA$.
For any $\bfa\in\bfA$, we denote by $P_{\bf a}:\bfT\to\bfT_{\bf a}$ a canonical projection and write
$$
\|\tau\|_{\bf a}:=\|P_{\bf a}\tau\|_{\bf a},\qquad\tau\in \bfT
$$
by abuse of notation.
\end{defi}

Obviously, the regularity structure is a particular case such that $\bfA\subset\bbR\times\{\infty\}$.

\subsection{Models}

We define the space of Besov type models on the fixed regularity-integrability structure $\scT=(\bfA,\bfT,\bfG)$.
For any measurable functions $f$ on $\bbR^d$ taking values in a Banach space $(X,\|\cdot\|_X)$, we use the notation
$$
\|f\|_{L^p(w;X)}:=\big\|\|f(x)\|_X\big\|_{L_x^p(w)}
$$
for simplicity. For two Banach spaces $X$ and $Y$, we denote by $\mcL(X,Y)$ the Banach space of all continuous linear operators $X\to Y$.

\begin{defi}
Let $w$ be a $G$-controlled weight.
A \emph{smooth model} $M=(\Pi,\Gamma)$ is a pair of two families of continuous linear operators $\Pi=\{\Pi_x:\bfT\to C(\bbR^d)\}_{x\in\bbR^d}$ and $\Gamma=\{\Gamma_{xy}\}_{x,y\in\bbR^d}\subset \bfG$ with the following properties.
\begin{itemize}
\item[(1)]
(Algebraic conditions)
$\Pi_x\Gamma_{xy}=\Pi_y$, $\Gamma_{xx}=\id$, and $\Gamma_{xy}\Gamma_{yz}=\Gamma_{xz}$ for any $x,y,z\in\bbR^d$.
\item[(2)]
(Analytic conditions)
For any $\bfc\in\bbR\times[1,\infty]$,
\begin{align*}
\|\Pi\|_{\bfc,w}&:=\max_{\bfa\in\bfA,\, \bfa\prec\bfc}\,
\sup_{0<t\le1}
\Big(t^{-r(\bfa)/\ell}\big\|Q_t\big(x,\Pi_x(\cdot)\big)\big\|_{L_x^{i(\bfa)}(w;\bfT_{\bf a}^*)}\Big)
\\
&=\max_{\bfa\in\bfA,\, \bfa\prec\bfc}\,
\sup_{0<t\le1}
\Bigg(t^{-r(\bfa)/\ell}
\Bigg\|
\sup_{\tau\in \bfT_{\bf a}\setminus\{0\}}
\frac{|Q_t(x,\Pi_x\tau)|}{\|\tau\|_{\bf a}}
\Bigg\|_{L_x^{i(\bfa)}(w)}\Bigg)<\infty
\end{align*}
and
\begin{align*}
&\|\Gamma\|_{\bfc,w}:=\max_{\substack{\bfa,\bfb\in\bfA\\ \bfb\prec\bfa\prec\bfc}}\,
\sup_{h\in\bbR^d\setminus\{0\}}
\frac{\|\Gamma_{(x+h)x}\|_{L_x^{i(\bfaob)}(w;\mcL(\bfT_{\bf a},\bfT_{\bf b}))}}{\wco(h)\|h\|_\mfs^{r(\bfaob)}}\\
&=\max_{\substack{\bfa,\bfb\in\bfA\\ \bfb\prec\bfa\prec\bfc}}\,
\sup_{h\in\bbR^d\setminus\{0\}}\Bigg(
\frac{1}{\wco(h)\|h\|_\mfs^{r(\bfaob)}}
\Bigg\|
\sup_{\tau\in \bfT_{\bf a}\setminus\{0\}}
\frac{\|\Gamma_{(x+h)x}\tau\|_{\bf b}}{\|\tau\|_{\bf a}}
\Bigg\|_{L_x^{i(\bfaob)}(w)}\Bigg)<\infty.
\end{align*}
\end{itemize}
We write $\tri M\tri_{\bfc,w}:=\|\Pi\|_{\bfc,w}+\|\Gamma\|_{\bfc,w}$.
In addition, for any two smooth models $M^{(i)}=(\Pi^{(i)},\Gamma^{(i)})$ with $i\in\{1,2\}$, we define the pseudo-metrics
$$
\tri M^{(1)};M^{(2)}\tri_{\bfc,w}:=\|\Pi^{(1)}-\Pi^{(2)}\|_{\bfc,w}+\|\Gamma^{(1)}-\Gamma^{(2)}\|_{\bfc,w}
$$
by replacing $\Pi$ and $\Gamma$ above with $\Pi^{(1)}-\Pi^{(2)}$ and $\Gamma^{(1)}-\Gamma^{(2)}$ respectively.
Finally, we define the space $\scM_w(\scT)$ as the completion of the set of all smooth models, under the pseudo-metrics $\tri\cdot;\cdot\tri_{\bfc,w}$ for all $\bfc\in\bbR\times[1,\infty]$.
We call each element of $\scM_w(\scT)$ a \emph{model} for $\scT$. We still use the notation $M=(\Pi,\Gamma)$ to denote a generic model.
\end{defi}

Recall that $i(\bfaob)=\infty$ if $\bfa,\bfb\in\bbR\times\{\infty\}$. Therefore, in the regularity structure case $\bfA\subset\bbR\times\{\infty\}$, the above definition coincides with the original definition of models \cite[Definition 2.17]{Hai14} if we ignore the difference between local and global bounds.

\medskip

It is a subtle question in which space the operator ``$\Pi_x$" takes values for general $M=(\Pi,\Gamma)\in\scM_w(\scT)$.
Under some additional assumptions on weights, we can regard $\Pi_x$ as a continuous linear operator from $\bfT$ to a Besov space.

\begin{prop}\label{prop:whatisPixtau}
Let $\alpha_0\le0$ be a regularity of $\scT$.
Assume that there exist two $G$-controlled weights $w_1$ and $w_2$ such that
$$
\sup_{x\in\bbR^d}\big\{\|x\|_\mfs^n\,\wco(x)w_1(x)\big\}
+\sup_{x\in\bbR^d}\big\{\|x\|_\mfs^n\,\wco_1(x)w_2(x)\big\}
<\infty
$$
for any $n\ge0$,
and that $ww_1$ and $ww_2$ are also $G$-controlled. Then for almost every $x\in\bbR^d$, the map $\Pi_x$ is well-defined as a continuous linear operator from $\bfT_{\bfa}$ to $B_{i(\bfa),1}^{\alpha,Q}(ww_1)$ for any $\bfa\in\bfA$ and any $\alpha<\alpha_0$. More precisely, for any $\bfc\in\bbR\times[1,\infty]$ such that $\bfa\prec\bfc$ we have
$$
\|\Pi_x\|_{L_x^{i(\bfa)}\big(ww_2;\mcL\big(\bfT_\bfa,B_{i(\bfa),1}^{\alpha,Q}(ww_1)\big)\big)}
\lesssim\|\Pi\|_{\bfc,w}(1+\|\Gamma\|_{\bfc,w}).
$$
\end{prop}

\begin{proof}
By the density argument, it is sufficient to show the inequality for smooth models.
By the algebraic relations and Lemmas \ref{lem:weightedholder} and \ref{lem:weightedtranslation},
\begin{align*}
\big\|Q_t\big(y,\Pi_x(\cdot)\big)\big\|_{\mcL(\bfT_\bfa,L_y^{i(\bfa)}(ww_1))}
&=\sup_{\tau\in\bfT_\bfa\setminus\{0\}}\frac{\|Q_t(y,\Pi_y\Gamma_{yx}\tau)\|_{L_y^{i(\bfa)}(ww_1)}}{\|\tau\|_\bfa}\\
&\le\sum_{\bfb\preceq \bfa}\Bigg\|\big\|Q_t\big(y,\Pi_y(\cdot)\big)\big\|_{\bfT_{\bfb}^*}
\sup_{\tau\in\bfT_\bfa\setminus\{0\}}\frac{\|\Gamma_{yx}\tau\|_{\bfb}}{\|\tau\|_\bfa}\Bigg\|_{L_y^{i(\bfa)}(ww_1)}\\
&\le \sum_{\bfb\preceq\bfa}\big\|Q_t\big(y,\Pi_y(\cdot)\big)\big\|_{L_y^{i(\bfb)}(w;\bfT_\bfb^*)}
\big\|\|\Gamma_{yx}\|_{\mcL(\bfT_\bfa,\bfT_\bfb)}\big\|_{L_y^{i(\bfa\ominus\bfb)}(w_1)}\\
&\le \|\Pi\|_{\bfc,w}\sum_{\bfb\preceq\bfa}t^{r(\bfb)/\ell}\wco_1(x)
\big\|\|\Gamma_{(x+y)x}\|_{\mcL(\bfT_\bfa,\bfT_\bfb)}\big\|_{L_y^{i(\bfa\ominus\bfb)}(w_1)}.
\end{align*}
By H\"older's inequality and Fubini's theorem, we have
\begin{align*}
&\Big\|\big\|Q_t\big(y,\Pi_x(\cdot)\big)\big\|_{\mcL(\bfT_\bfa,L_y^{i(\bfa)}(ww_1))}\Big\|_{L_x^{i(\bfa)}(ww_2)}\\
&\le \|\Pi\|_{\bfc,w}\sum_{\bfb\preceq\bfa}t^{r(\bfb)/\ell}\|\wco_1\|_{L^{i(\bfb)}(w_2)}
\Big\|\big\|\|\Gamma_{(x+y)x}\|_{\mcL(\bfT_\bfa,\bfT_\bfb)}\big\|_{L_y^{i(\bfaob)}(w_1)}\Big\|_{L_x^{i(\bfaob)}(w)}\\
&=\|\Pi\|_{\bfc,w}\sum_{\bfb\preceq\bfa}t^{r(\bfb)/\ell}\|\wco_1\|_{L^{i(\bfb)}(w_2)}
\Big\|\big\|\|\Gamma_{(x+y)x}\|_{\mcL(\bfT_\bfa,\bfT_\bfb)}\big\|_{L_x^{i(\bfaob)}(w)}\Big\|_{L_y^{i(\bfaob)}(w_1)}\\
&\lesssim\|\Pi\|_{\bfc,w}(1+\|\Gamma\|_{\bfc,w})\sum_{\bfb\preceq\bfa}t^{r(\bfb)/\ell}\big\|1+\wco(y)\|y\|_\mfs^{r(\bfaob)}\big\|_{L_y^{i(\bfaob)}(w_1)}\\
&\lesssim\|\Pi\|_{\bfc,w}(1+\|\Gamma\|_{\bfc,w})\, t^{\alpha_0/\ell}
\end{align*}
for $t\in(0,1]$. Note that $w^*\ge1$ and $\|1\|_{L^p(w_1)}<\infty$ for any $p\in[1,\infty]$ by an assumption on $w_1$.
Finally, by the definition of Besov norms,
\begin{align*}
&\Big\|\|\Pi_x\|_{\mcL\big(\bfT_\bfa,B_{i(\bfa),1}^{\alpha,Q}(ww_1)\big)}\Big\|_{L_x^{i(\bfa)}(ww_2)}\\
&\le\bigg\|
\big\|Q_1\big(y,\Pi_x(\cdot)\big)\big\|_{\mcL(\bfT_\bfa,L_y^{i(\bfa)}(ww_1))}
+\int_0^1t^{-\alpha/\ell}\big\|Q_t\big(y,\Pi_x(\cdot)\big)\big\|_{\mcL(\bfT_\bfa,L_y^{i(\bfa)}(ww_1))}\frac{dt}t
\bigg\|_{L_x^{i(\bfa)}(ww_2)}\\
&\lesssim\|\Pi\|_{\bfc,w}(1+\|\Gamma\|_{\bfc,w})\bigg(1+\int_0^1t^{(\alpha_0-\alpha)/\ell}\frac{dt}t\bigg)
\lesssim\|\Pi\|_{\bfc,w}(1+\|\Gamma\|_{\bfc,w}).
\end{align*}
\end{proof}

\begin{rem}\label{rem:whatisknownforZ}
Without additional weights as above, we only know that $Q_t\big(x-h,\Pi_x(\cdot)\big)$ is defined for any $h\in\bbR^d$ as an element of $L_x^{i(\bfa)}(w^2;\bfT_\bfa^*)$. Indeed, for any smooth model and for any $\bfa\prec\bfc$, by Lemmas \ref{lem:weightedholder} and \ref{lem:weightedtranslation} we have
\begin{align}\label{eq:Q(x-h,Pix)}
\begin{aligned}
\big\|Q_t\big(x-h,\Pi_x(\cdot)&\big)\big\|_{L_x^{i(\bfa)}(w^2;\bfT_\bfa^*)}
=\big\|Q_t\big(x-h,\Pi_{x-h}\Gamma_{(x-h)x}(\cdot)\big)\big\|_{L_x^{i(\bfa)}(w^2;\bfT_\bfa^*)}\\
&\le\sum_{\bfb\preceq\bfa}\Big\|\big\|Q_t\big(x-h,\Pi_{x-h}(\cdot)\big)\big\|_{\bfT_\bfb^*}\|\Gamma_{(x-h)x}\|_{\mcL(\bfT_\bfa,\bfT_\bfb)}\Big\|_{L_x^{i(\bfa)}(w^2)}\\
&\le \wco(h)\sum_{\bfb\preceq\bfa}\big\|Q_t\big(x,\Pi_x(\cdot)\big)\big\|_{L_x^{i(\bfb)}(w;\bfT_\bfb^*)}\|\Gamma_{(x-h)x}\|_{L_x^{i(\bfaob)}(w;\mcL(\bfT_\bfa,\bfT_\bfb))}\\
&\le \big(\wco(h)\big)^2\|\Pi\|_{\bfc,w}(1+\|\Gamma\|_{\bfc,w})\sum_{\bfb\preceq\bfa}t^{r(\bfb)/\ell}\|h\|_\mfs^{r(\bfaob)}<\infty.
\end{aligned}
\end{align}
Moreover, by the density argument we also have the semigroup property
$$
Q_t(x,\Pi_x\tau)=\int_{\bbR^d}Q_{t-s}(x,x-h)Q_s(x-h,\Pi_x\tau)dh,\qquad 0<s<t
$$
for any models. These properties are used to prove the reconstruction theorem.
\end{rem}

\subsection{Modelled distributions}

We close this section with the definition of Besov type modelled distributions and their reconstructions.
We fix two $G$-controlled weights $w$ and $v$ such that $wv$ is also $G$-controlled.

\begin{defi}\label{def:besovMD}
Let $M=(\Pi,\Gamma)\in\scM_{w}(\scT)$.
For any $\bfc\in\bbR\times[1,\infty]$, we define $\mcD_{v}^\bfc(\Gamma)$ as the space of all functions $f:\bbR^d\to \bfT_{\prec \bfc}:=\bigoplus_{\bfa\in \bfA,\, \bfa\prec \bfc}\bfT_\bfa$ such that
\begin{align*}
\lp f\rp_{\bfc,v}&:=\max_{\bfa\prec \bfc}\|f\|_{L^{i(\bfcoa)}(v;\bfT_\bfa)}<\infty,\\
\| f\|_{\bfc,v}^\Gamma&:=\max_{\bfa\prec\bfc}
\sup_{h\in\bbR^d\setminus\{0\}}
\frac{\|\Delta_{x;h}^\Gamma f\|_{L_x^{i(\bfcoa)}(v;\bfT_\bfa)}}{\vco(h)\|h\|_\mfs^{r(\bfcoa)}}<\infty,
\end{align*}
where $\Delta_{x;h}^\Gamma f:=f(x-h)-\Gamma_{(x-h)x}f(x)$.
We write $\tri f\tri_{\bfc,v}^\Gamma:=\lp f\rp_{\bfc,v}+\|f\|_{\bfc,v}^\Gamma$.
We call each element of $\mcD_v^{\bfc}(\Gamma)$ a \emph{modelled distribution}.

In addition, for any two models $M^{(i)}=(\Pi^{(i)},\Gamma^{(i)})\in\scM_{w}(\scT)$ and modelled distributions $f^{(i)}\in\mcD_{v}^\bfc(\Gamma^{(i)})$ with $i\in\{1,2\}$, we define $\tri f^{(1)};f^{(2)}\tri_{\bfc,v}^{\Gamma^{(1)};\Gamma^{(2)}}:=\lp f^{(1)}-f^{(2)}\rp_{\bfc,v}+\|f^{(1)};f^{(2)}\|_{\bfc,v}^{\Gamma^{(1)};\Gamma^{(2)}}$ by
\begin{align*}
\lp f^{(1)}-f^{(2)}\rp_{\bfc,v}&:=\max_{\bfa\prec \bfc}\|f^{(1)}-f^{(2)}\|_{L^{i(\bfcoa)}(v;\bfT_\bfa)},\\
\| f^{(1)};f^{(2)}\|_{\bfc,v}^{\Gamma^{(1)};\Gamma^{(2)}}&:=\max_{\bfa\prec \bfc}
\sup_{h\in\bbR^d\setminus\{0\}}
\frac{\|\Delta_{x;h}^{\Gamma^{(1)}} f^{(1)}-\Delta_{x;h}^{\Gamma^{(2)}}f^{(2)}\|_{L_x^{i(\bfcoa)}(v;\bfT_\bfa)}}{\vco(h)\|h\|_\mfs^{r(\bfcoa)}}.
\end{align*}
We omit the symbol ``$\,\Gamma^{(1)};\Gamma^{(2)}$" below for simplicity.
\end{defi}

\begin{defi}\label{def:besovreconst}
Let $\alpha_0\le0$ be a regularity of $\scT$ and let $\bfc\in\bbR\times[1,\infty]$.
For any $M=(\Pi,\Gamma)\in\scM_{w}(\scT)$ and $f\in\mcD_{v}^\bfc(\Gamma)$, we call any $\Lambda\in B_{i(\bfc)}^{\alpha_0,Q}(wv)$ satisfying
\begin{align*}
\lb\Lambda\rb_{\bfc,wv}^{\Pi,f}:=\sup_{0<t\le1}t^{-r(\bfc)/\ell}\big\|Q_t(x,\Lambda)-Q_t\big(x,\Pi_xf(x)\big)\big\|_{L_x^{i(\bfc)}(wv)}<\infty
\end{align*}
a \emph{reconstruction} of $f$ for $M$.
Furthermore, for any models $M^{(i)}=(\Pi^{(i)},\Gamma^{(i)})\in\scM_{w}(\scT)$, modelled distributions $f^{(i)}\in\mcD_{v}^e(\Gamma^{(i)})$, and any reconstructions $\Lambda^{(i)}\in B_{i(\bfc)}^{\alpha_0,Q}(wv)$ of $f^{(i)}$ for $M^{(i)}$ with $i\in\{1,2\}$, we define
\begin{align*}
\lb\Lambda^{(1)};\Lambda^{(2)}\rb_{\bfc,wv}^{\Pi^{(1)},f^{(1)};\Pi^{(2)},f^{(2)}}
:=\sup_{0<t\le1}t^{-r(\bfc)/\ell}\Big\|&\big\{Q_t\big(x,\Lambda^{(1)}\big)-Q_t\big(x,\Pi_x^{(1)}f^{(1)}(x)\big)\big\}\\
&-\big\{Q_t\big(x,\Lambda^{(2)}\big)-Q_t\big(x,\Pi_x^{(2)}f^{(2)}(x)\big)\big\}\Big\|_{L_x^{i(\bfc)}(wv)}.
\end{align*}
We also omit the symbol ``$\,\Pi^{(1)},f^{(1)};\Pi^{(2)},f^{(2)}$" below for simplicity.
\end{defi}

\begin{rem}\label{rem:L-Pixfx}
It seems more natural to write ``$\, Q_t\big(x,\Lambda-\Pi_xf(x)\big)$", but we split it into two terms here to avoid the subtle question of what ``$\, \Pi_xf(x)$" is (see Proposition \ref{prop:whatisPixtau}).
Since $Q_t\big(x,\Pi_x(\cdot)\big)$ is well-defined as an element of $L_x^{i(\bfa)}(w;\bfT_\bfa^*)$,
we can define the quantity $Q_t\big(x,\Pi_xf(x)\big)$ by inserting $f(x)$ into the operator $Q_t\big(x,\Pi_x(\cdot)\big)$. See also the calculations at the beginning of the proof of Theorem \ref{thm:besovreconstruction}.
We can also define $Q_t\big(x-h,\Pi_xf(x)\big)$ for any $h\in\bbR^d$ by Remark \ref{rem:whatisknownforZ}.
\end{rem}

\section{Reconstruction theorem}\label{section:reconstruction}

In this section, we fix a regularity-integrability structure $\scT=(\bfA,\bfT,\bfG)$ of regularity $\alpha_0\le0$, and also fix $G$-controlled weights $w$ and $v$ such that $wv$ is also $G$-controlled.

\begin{thm}\label{thm:besovreconstruction}
Let $\bfc\in(0,\infty)\times[1,\infty]$.
Then for any $M=(\Pi,\Gamma)\in\scM_{w}(\scT)$ and $f\in\mcD_{v}^\bfc(\Gamma)$, there exists a unique reconstruction $\mcR f$ of $f$ for $M$ and it holds that
\begin{align}
\label{besovreconstruction}
\|\mcR f\|_{B_{i(\bfc)}^{\alpha_0,Q}(wv)}&\lesssim\|\Pi\|_{\bfc,w}\tri f\tri_{\bfc,v}^\Gamma,\\
\label{besovreconstructioncharacterize}
\lb\mcR f\rb_{\bfc,wv}^{\Pi,f}&\lesssim \|\Pi\|_{\bfc,w}\|f\|_{\bfc,v}^\Gamma.
\end{align}
Moreover, there is an affine function $C_R>0$ of $R>0$ such that
\begin{align*}
\|\mcR f^{(1)}-\mcR f^{(2)}\|_{B_{i(\bfc)}^{\alpha_0,Q}(wv)}&\le C_R\big(\|\Pi^{(1)}-\Pi^{(2)}\|_{\bfc,w}+\tri f^{(1)};f^{(2)}\tri_{\bfc,v}\big),\\
\lb\mcR f^{(1)};\mcR f^{(2)}\rb_{\bfc,wv}&\le C_R\big(\|\Pi^{(1)}-\Pi^{(2)}\|_{\bfc,w}+\| f^{(1)};f^{(2)}\|_{\bfc,v}\big)
\end{align*}
for any $M^{(i)}=(\Pi^{(i)},\Gamma^{(i)})\in\scM_{w}(\scT)$ and $f^{(i)}\in\mcD_{v}^\bfc(\Gamma^{(i)})$ with $i\in\{1,2\}$ such that $\tri M^{(i)}\tri_{\bfc,w}\le R$ and $\tri f^{(i)}\tri_{\bfc,v}\le R$.
\end{thm}

\begin{rem}
The original reconstruction theorem \cite[Theorem 3.10]{Hai14} was extended to different types of norms \cite{HL17, HR20, BL22}.
Hairer and Labb\'e \cite{HL17} proved a reconstruction theorem for $B_{\infty,\infty}$ type models and $B_{p,q}$ type modelled distributions. Their result is a special case of Theorem \ref{thm:besovreconstruction} if we ignore $q$-exponents.
Broux and Lee \cite{BL22} proved Besov reconstruction theorem at the level of ``coherent germ", which was the notion introduced by Caravenna and Zambotti \cite{CZ20} to reformulate the reconstruction theorem without using regularity structures.
As seen in the following proof, our situation is contained in \cite{BL22} as a special case because $\{F_x:=\Pi_xf(x)\}_{x\in\bbR^d}$ is a coherent germ.
However, the detailed regularity-integrability structure is effectively used in the paper \cite{BPHZ}.
As for the different norm, Hensel and Rosati \cite{HR20} proved Triebel--Lizorkin type reconstruction theorem for $F_{\infty,\infty}$ type models and $F_{p,q}$ type modelled distributions.
\end{rem}

\begin{proof}
It is sufficient to show the bounds \eqref{besovreconstruction} and \eqref{besovreconstructioncharacterize} for a single model and modelled distribution. The proofs of the local Lipschitz estimates are simple modifications.
For $t>0$ and $0<s\le t\wedge1$, we define the functions
\begin{align*}
\mcR_s^t f(x):=
\left\{
\begin{aligned}
&\int_{\bbR^d}Q_{t-s}(x,y)Q_s\big(y,\Pi_yf(y)\big)dy,\quad&&s<t,\\
&Q_t\big(x,\Pi_xf(x)\big),&&s=t.
\end{aligned}
\right.
\end{align*}
By Lemma \ref{lem:weightedholder}, we have
\begin{align*}
\big\|Q_s\big(y,\Pi_yf(y)\big)\big\|_{L_y^{i(\bfc)}(wv)}
&\le\sum_{\bfa\prec \bfc}\Big\|\big\|Q_s\big(y,\Pi_y(\cdot)\big)\big\|_{\bfT_\bfa^*}\|P_\bfa f(y)\|_{\bfT_\bfa}\Big\|_{L_y^{i(\bfc)}(wv)}\\
&\le\sum_{\bfa\prec \bfc}\big\|Q_s\big(y,\Pi_y(\cdot)\big)\big\|_{L_y^{i(\bfa)}(w;\bfT_\bfa^*)}\|f\|_{L^{i(\bfcoa)}(v;\bfT_\bfa)}\\
&\le \|\Pi\|_{\bfc,w}\lp f\rp_{\bfc,v}\sum_{\bfa\prec \bfc}s^{r(\bfa)/\ell}
\lesssim\|\Pi\|_{\bfc,w}\lp f\rp_{\bfc,v}\, s^{\alpha_0/\ell}.
\end{align*}
Hence by Proposition \ref{prop:operatorQt}, we have $\mcR_s^tf\in L_c^{i(\bfc)}(wv)$ and $\|\mcR_s^tf\|_{L^{i(\bfc)}(wv)}\lesssim\|\Pi\|_{\bfc,w}\lp f\rp_{\bfc,v}\, s^{\alpha_0/\ell}$.
We separate the proof into five steps.

\medskip

\noindent
{\bf (1) Coherence property.}\ Set $F_x:=\Pi_xf(x)$. (This is an abuse of notation as mentioned in Remark \ref{rem:L-Pixfx}, but it does not cause a serious problem because it always appears in the form $Q_t(x-h,F_x)$.) By Lemmas \ref{lem:weightedholder} and \ref{lem:weightedtranslation}, we have
\begin{align*}
\big\|Q_t(x-h,F_x-F_{x-h})\big\|_{L_x^{i(\bfc)}(wv)}
&=\big\|Q_t\big(x-h,\Pi_{x-h}\big\{\Gamma_{(x-h)x}f(x)-f(x-h)\big\}\big)\big\|_{L_x^{i(\bfc)}(wv)}\\
&\le\sum_{\bfa\prec \bfc}\big\|Q_t\big(x-h,\Pi_{x-h}(\cdot)\big)\big\|_{L_x^{i(\bfa)}(w;\bfT_\bfa^*)}\|\Delta_{x;h}^\Gamma f\|_{L_x^{i(\bfcoa)}(v;\bfT_\bfa)}
\\
&\le \|\Pi\|_{\bfc,w}\|f\|_{\bfc,v}^\Gamma\,\wco(h)\sum_{\bfa\prec \bfc}t^{r(\bfa)/\ell}\vco(h)\|h\|_{\mfs}^{r(\bfcoa)}.
\end{align*}

\noindent
{\bf (2) Convergence as $s\downarrow0$.}\ 
By the semigroup property of $\{Q_t\}_{t>0}$, for $0<u<s<t\wedge1$,
\begin{align*}
|\mcR_s^tf(x)-\mcR_u^tf(x)|
&=\bigg|\int_{(\bbR^d)^2}Q_{t-s}(x,y)Q_{s-u}(y,y-h)Q_u(y-h,F_y-F_{y-h})dydh\bigg|\\
&\lesssim\int_{(\bbR^d)^2}G_{t-s}(x-y)G_{s-u}(h)\big|Q_u(y-h,F_y-F_{y-h})\big|dydh.
\end{align*}
By applying Lemma \ref{lem:weightedyoung} to the convolution with respect to $y$,
\begin{align*}
\|\mcR_u^tf-\mcR_s^tf\|_{L^{i(\bfc)}(wv)}
&\lesssim\int_{\bbR^d}G_{s-u}(h)\big\|Q_u(y-h,F_{y}-F_{y-h})\big\|_{L_y^{i(\bfc)}(wv)}dh\\
&\le\|\Pi\|_{\bfc,w}\|f\|_{\bfc,v}^\Gamma\sum_{\bfa\prec\bfc}u^{r(\bfa)/\ell}\int_{\bbR^d}\|h\|_{\mfs}^{r(\bfcoa)}(\wco\vco)(h)G_{s-u}(h)dh\\
&\lesssim
\|\Pi\|_{\bfc,w}\|f\|_{\bfc,v}^\Gamma\sum_{\bfa\prec\bfc}u^{r(\bfa)/\ell}(s-u)^{r(\bfcoa)/\ell}.
\end{align*}
In the last inequality, we use the property \eqref{weightintegrable} for $\wco\vco$.
Since $r(\bfa)+r(\bfcoa)=r(\bfc)$, we have for any $u\in[s/2,s)$,
\begin{align}\label{eq:besovreconstructioncauchy}
\|\mcR_u^tf-\mcR_s^tf\|_{L^{i(\bfc)}(wv)}\lesssim\|\Pi\|_{\bfc,w}\|f\|_{\bfc,v}^\Gamma\, s^{r(\bfc)/\ell}.
\end{align}
Also for $u\in(0,s/2)$, by taking $n\in\bbN$ such that $u\in[s/2^{n+1},s/2^n)$, we have
\begin{align*}
\|\mcR_u^tf-\mcR_s^tf\|_{L^{i(\bfc)}(wv)}
&\le\sum_{m=0}^{n-1}\|\mcR_{s/2^m}^tf-\mcR_{s/2^{m+1}}^tf\|_{L^{i(\bfc)}(wv)}+\|\mcR_{s/2^n}^tf-\mcR_u^tf\|_{L^{i(\bfc)}(wv)}\\
&\lesssim\|\Pi\|_{\bfc,w}\|f\|_{\bfc,v}^\Gamma\bigg\{\sum_{m=0}^{n-1}\bigg(\frac{s}{2^m}\bigg)^{r(\bfc)/\ell}+\bigg(\frac{s}{2^n}\bigg)^{r(\bfc)/\ell}\bigg\}\\
&\lesssim\|\Pi\|_{\bfc,w}\|f\|_{\bfc,v}^\Gamma\, s^{r(\bfc)/\ell}.
\end{align*}
Moreover, the same bound for the case $s=t\le1$ can be obtained by a similar argument.
In the end, the bound \eqref{eq:besovreconstructioncauchy} holds for any $0<u<s\le t\wedge1$.
Since $r(\bfc)>0$, this implies that $\{\mcR_s^tf\}_{0<s\le t\wedge1}$ is Cauchy in $L_c^{i(\bfc)}(wv)$ as $s\downarrow0$. We denote its limit by
$$
\mcR_0^tf:=\lim_{s\downarrow0}\mcR_s^tf.
$$

\noindent
{\bf (3) Uniform bounds.}\ Combining the Cauchy property \eqref{eq:besovreconstructioncauchy} with the initial bound
$$
\|\mcR_{t\wedge1}^tf\|_{L^{i(\bfc)}(wv)}\lesssim\|\Pi\|_{\bfc,w}\lp f\rp_{\bfc,v}(t\wedge1)^{\alpha_0/\ell}
$$
obtained at the beginning of the proof, we have
\begin{align}\label{eq:besovreconstructuniform}
\|\mcR_0^tf\|_{L^{i(\bfc)}(wv)}
\lesssim\|\Pi\|_{\bfc,w}\tri f \tri_{\bfc,v}^\Gamma(t\wedge1)^{\alpha_0/\ell}.
\end{align}
Incidentally, we have the identity
$$
Q_s\mcR_u^tf=\mcR_u^{t+s}f,\qquad 0<u\le t\wedge1,\ s>0
$$
from the semigroup property. Letting $u\downarrow0$, we have
\begin{align}\label{eq:besovreconstructsemigroupR0}
Q_s\mcR_0^tf=\mcR_0^{t+s}f,\qquad t,s>0.
\end{align}
By \eqref{eq:besovreconstructsemigroupR0} and the bound \eqref{eq:besovreconstructuniform}, we have the estimate of $\mcR_0^tf$ in the Besov norm
\begin{align}\label{eq:besovreconstructBesovbdd}
\|\mcR_0^tf\|_{B_{i(\bfc)}^{\alpha_0,Q}(wv)}
&=\sup_{0<s\le1}s^{-\alpha_0/\ell}\|\mcR_0^{t+s}f\|_{L^{i(\bfc)}(wv)}
\lesssim\|\Pi\|_{\bfc,w}\tri f \tri_{\bfc,v}^\Gamma.
\end{align}

\medskip

\noindent
{\bf (4) Convergence as $t\downarrow0$.}\ By the semigroup property \eqref{eq:besovreconstructsemigroupR0}, the uniform bound \eqref{eq:besovreconstructBesovbdd}, and Lemma \ref{lem:timecontinuity}, we have for any $\varepsilon\in(0,\ell]$ and any $0<s<t$,
\begin{align*}
\|\mcR_0^tf-\mcR_0^sf\|_{B_{i(\bfc)}^{\alpha_0-\varepsilon,Q}(wv)}
&=\|(Q_{t-s}-\id)\mcR_0^sf\|_{B_{i(\bfc)}^{\alpha_0-\varepsilon,Q}(wv)}
\lesssim(s-r)^{\varepsilon/\ell}\|\Pi\|_{\bfc,w}\tri f \tri_{\bfc,v}^\Gamma.
\end{align*}
This implies that $\{\mcR_0^tf\}_{t\in(0,1]}$ is Cauchy in $B_{i(\bfc)}^{\alpha_0-\varepsilon,Q}(wv)$ as $t\downarrow0$. We denote its limit by
$$
\mcR f:=\lim_{t\downarrow0}\mcR_0^tf.
$$
Incidentally, by letting $t\downarrow0$ in \eqref{eq:besovreconstructsemigroupR0}, we have 
$$
Q_s\mcR f=\mcR_0^sf,\qquad s>0.
$$
Combining this with the uniform bound \eqref{eq:besovreconstructuniform}, we have that $\mcR f$ actually belongs to $B_{i(\bfc)}^{\alpha_0,Q}(wv)$ and the result \eqref{besovreconstruction} follows.
On the other hand, by letting $u\downarrow0$ and $s=t$ in \eqref{eq:besovreconstructioncauchy}, we have the result \eqref{besovreconstructioncharacterize}.

\medskip

\noindent
{\bf (5) Uniqueness.}\ Let $\Lambda,\Lambda'\in B_{i(\bfc)}^{\alpha_0,Q}(wv)$ be reconstructions of $f$ for $M$.
From the definition of reconstruction, $g:=\Lambda-\Lambda'\in B_{i(\bfc)}^{\alpha_0,Q}(wv)$ satisfies
\begin{align*}
\|Q_tg\|_{L^{i(\bfc)}(wv)}
&\le\big\|Q_t(x,\Lambda)-Q_t\big(x,\Pi_xf(x)\big)\big\|_{L_x^{i(\bfc)}(wv)}+\big\|Q_t(x,\Lambda')-Q_t\big(x,\Pi_xf(x)\big)\big\|_{L_x^{i(\bfc)}(wv)}\\
&\lesssim t^{r(\bfc)/\ell}.
\end{align*}
Then by using Lemma \ref{lem:timecontinuity} again, we have that for any $\varepsilon\in(0,\ell]$
\begin{align*}
\|g\|_{B_{i(\bfc)}^{\alpha_0-\varepsilon,Q}(wv)}&\le\|(Q_t-\id)g\|_{B_{i(\bfc)}^{\alpha_0-\varepsilon,Q}(wv)}+\|Q_tg\|_{B_{i(\bfc)}^{\alpha_0-\varepsilon,Q}(wv)}\\
&\lesssim t^{\varepsilon/\ell}\|g\|_{B_{i(\bfc)}^{\alpha_0,Q}(wv)}+\|Q_tg\|_{L^{i(\bfc)}(wv)}
\lesssim t^{(\varepsilon\wedge r(\bfc))/\ell}.
\end{align*}
Since $r(\bfc)>0$, we have $g=0$ in $B_{i(\bfc)}^{\alpha_0-\varepsilon,Q}(wv)$ by taking the limit $t\downarrow0$. By Proposition \ref{prop:abinjection}, this implies $g=0$ in $B_{i(\bfc)}^{\alpha_0,Q}(wv)$, and thus $\Lambda=\Lambda'$.
\end{proof}

\section{Multilevel Schauder estimate}\label{section:schauder}

In this section, we consider the integral operator of the form $f\mapsto \int_{\bbR^d}K(\cdot,x)f(x)dx$.
The convolution with Green function of Laplacian and the spacetime convolution with heat kernel are typical examples. We lift such an integral operator to the operator $\mcK$ acting on modelled distributions and prove its continuity (Theorem \ref{thm:besovschauder}).

\subsection{Regularizing kernels}

We have in mind the integral kernel $L(x,y)$ typically singular at the diagonal $\{x=y\}$, but precisely we consider its ``rough part" $K(x,y)$.
As in \cite[Lemma 5.5]{Hai14}, Hairer considered a decomposition $L=K+R$, where $R$ is a ``smooth part" which sends any distributions into sufficiently smooth functions, and a rough part $K$ of $L$ can be decomposed into the sum $\sum_{n=0}^\infty K_n$ of locally supported smooth functions $K_n$ with good scaling properties.
In this paper, we instead consider an integral representation $K=\int_0^1K_tdt$ of $K$ by smooth functions $K_t$. We impose a restrictive assumption (Definition \ref{asmp2}-\ref{asmp2:convolutionwithQ} below) for the convolution of $K_t$ and $Q_s$ instead of generality, but it simplifies the proof of multilevel Schauder estimate.

\begin{defi}\label{asmp2}
Let $\bar\beta>0$.
A \emph{$\bar\beta$-regularizing (integral) kernel admissible for $\{Q_t\}_{t>0}$} is a family of continuous functions $\{K_t:\bbR^d\times\bbR^d\to\bbR\}_{t>0}$ which satisfies the following properties for some constants $\delta>0$ and $C_K>0$.
\begin{enumerate}
\renewcommand{\theenumi}{(\roman{enumi})}
\renewcommand{\labelenumi}{(\roman{enumi})}
\item\label{asmp2:convolutionwithQ}
(Convolution with $Q$)
For any $0<s<t$ and $x,y\in\bbR^d$,
$$
\int_{\bbR^d}K_{t-s}(x,z)Q_s(z,y)dz=K_t(x,y).
$$
\item\label{asmp2:gauss}
(Upper estimate)
For any ${\bf k}\in\bbN^d$ with $|{\bf k}|_\mfs<\delta$, the ${\bf k}$-th partial derivative of $K_t(x,y)$ with respect to $x$ exists, and we have for any $t>0$ and $x,y\in\bbR^d$,
$$
|\partial_x^{\bf k}K_t(x,y)|\le C_Kt^{(\bar\beta-|{\bf k}|_\mfs)/\ell-1}G_t(x-y).
$$
\item\label{asmp2:holder}
(H\"older continuity)
For any ${\bf k}\in\bbN^d$ with $|{\bf k}|_\mfs<\delta$, any $t>0$ and $x,y,h\in\bbR^d$ with $\|h\|_{\mfs}\le t^{1/\ell}$,
\begin{align*}
&\bigg|\partial_x^{\bf k}K_t(x+h,y)-\sum_{|{\bf l}|_\mfs<\delta-|{\bf k}|_\mfs}\frac{h^{\bf l}}{{\bf l}!}\partial_x^{{\bf k}+{\bf l}}K_t(x,y)\bigg|
\le C_K\|h\|_\mfs^{\delta-|{\bf k}|_\mfs}\,  t^{(\bar\beta-\delta)/\ell-1}G_t(x-y).
\end{align*}
\end{enumerate}
\end{defi}

\begin{rem}\label{asmp2'}
The property \ref{asmp2:holder} still holds if we replace $\delta$ with arbitrary $\varepsilon\in(0,\delta)$.
To see this, we have only to decompose
\begin{align*}
&\partial_x^{\bf k}K_t(x+h,y)-\sum_{|{\bf l}|_\mfs<\varepsilon-|{\bf k}|_\mfs}\frac{h^{\bf l}}{{\bf l}!}\partial_x^{{\bf k}+{\bf l}}K_t(x,y)\\
&=\bigg(\partial_x^{\bf k}K_t(x+h,y)-\sum_{|{\bf l}|_\mfs<\delta-|{\bf k}|_\mfs}\frac{h^{\bf l}}{{\bf l}!}\partial_x^{{\bf k}+{\bf l}}K_t(x,y)\bigg)+\sum_{\varepsilon-|{\bf k}|_{\mfs}\le|{\bf l}|_{\mfs}<\delta-|{\bf k}|_{\mfs}}\frac{h^{\bf l}}{{\bf l}!}\partial_x^{{\bf k}+{\bf l}}K_t(x,y)
\end{align*}
and use properties \ref{asmp2:gauss}, \ref{asmp2:holder}, and the condition $\|h\|_{\mfs}\le t^{1/\ell}$.
\end{rem}

\begin{exam}
Let $\{Q_t\}_{t>0}$ be a $G$-type semigroup generated by the parabolic operator \eqref{exam:parabolicop} in Example \ref{exam:FS}. An example of admissible regularizing kernels is given by
\begin{align}\label{ex:K=DQ}
K_t(x,y)=\sum_{|{\bf k}|_\mfs\le \ell_1}b_{\bf k}(x)\partial_x^{\bf k}Q_t(x,y)
\end{align}
with an exponent $\ell_1<\ell$ and bounded H\"older continuous coefficients $b_{\bf k}(x)$. Then $\{K_t\}_{t>0}$ is $(\ell-\ell_1)$-regularizing.
The exponent $\delta$ depends on the H\"older regularity of the coefficients $b_{\bf k}$ and $a_{\bf k}$ in the operator \eqref{exam:parabolicop}. See \cite[Appendix A]{BHK22} for details.

\medskip

There are two typical examples.
\begin{itemize}
\item Let $\mfs=(1,1,\dots,1)$, $\ell=2$, and $P(\partial_x)=\Delta-1$ in \eqref{exam:parabolicop}. The corresponding $Q_t$ is the heat semigroup $e^{t(\Delta-1)}$. Then the inverse operator $(1-\Delta)^{-1}$ has the representation
$$
(1-\Delta)^{-1}=-\int_0^\infty Q_tdt=-\int_0^1Q_tdt+Q_1(1-\Delta)^{-1}.
$$
Since $Q_1(1-\Delta)^{-1}$ has a sufficient regularization effect, the rough part of $(1-\Delta)^{-1}$ is represented as the integral $\int_0^1K_tdt$ with $K_t=-Q_t$.
Since this $K_t$ is of the form \eqref{ex:K=DQ} with $\ell_1=0$, the regularizing order is $\bar\beta=2$. Moreover, since $K_t$ is smooth, we can choose arbitrary large $\delta>0$.
\item Let $\mfs=(2,1,\dots,1)$, $\ell=4$, and $P(\partial_x)=\partial_{x_1}^2-(\Delta_{x'}-1)^2$ in \eqref{exam:parabolicop}. Denote by $Q_t=e^{tP(\partial_x)}$ the corresponding heat semigroup. Then the inverse $\big(\partial_{x_1}-(\Delta_{x'}-1)\big)^{-1}$ of the parabolic operator (considered in $(x_1,x')\in\bbR\times\bbR^{d-1}$) has the representation
\begin{align*}
\big(\partial_{x_1}-(\Delta_{x'}-1)\big)^{-1}&=\big(\partial_{x_1}+(\Delta_{x'}-1)\big)\big(\partial_{x_1}^2-(\Delta_{x'}-1)^2\big)^{-1}\\
&=\big(\partial_{x_1}+(\Delta_{x'}-1)\big)\int_0^1Q_tdt-\big(\partial_{x_1}+(\Delta_{x'}-1)\big)Q_1\big(P(\partial_x)\big)^{-1}.
\end{align*}
Therefore, the rough part of $\big(\partial_{x_1}-(\Delta_{x'}-1)\big)^{-1}$ is also represented as the integral $\int_0^1K_tdt$ with $K_t=\big(\partial_{x_1}+(\Delta_{x'}-1)\big)Q_t$.
Since $\ell_1=2$ in this case, the regularizing order is $\bar\beta=4-2=2$, and we can choose arbitrary large $\delta>0$.
\end{itemize}
\end{exam}

The following result clarifies the meaning of regularizing kernels.
Let $w$ be a $G$-controlled weight. For any $\alpha>0$ and $p\in[1,\infty]$, we define $B_p^{\alpha}(w)$ as the space of all measurable functions $f$ such that, there is a family of measurable  functions $\{\partial^{\bf k}f\}_{|{\bf k}|_\mfs<\alpha}$ satisfying $\partial^{\bf 0}f=f$ and
\begin{align*}
\|T_{x;h}^{\alpha-|{\bf k}|_\mfs}(\partial^{\bf k}f)\|_{L_x^p(w)}
&:=\bigg\|(\partial^{\bf k}f)(x-h)-\sum_{|{\bf l}|_\mfs<\alpha-|{\bf k}|_\mfs}\frac{(-h)^{\bf l}}{{\bf l}!}(\partial^{{\bf k}+{\bf l}}f)(x)\bigg\|_{L_x^p(w)}\\
&\lesssim w^*(h)\|h\|_\mfs^{\alpha-|{\bf k}|_\mfs}
\end{align*}
for any $|{\bf k}|_\mfs<\alpha$.
In addition, recall the definition of $\bbN[\mfs]$ in Section \ref{subsection:notation}.

\begin{lem}\label{lem:whyregularize}
Let $\{K_t\}_{t>0}$ be a $\bar\beta$-regularizing kernel admissible for $\{Q_t\}_{t>0}$.
For any function $f\in L^p(w)$ and $|{\bf k}|_{\mfs}<\delta$, we define
$$
(\partial^{\bf k}K_tf)(x):=:\partial^{\bf k}K_t(x,f):=\int_{\bbR^d}\partial_x^{\bf k}K_t(x,y)f(y)dy,\qquad
K_tf:=\partial^{\bf 0}K_tf.
$$
Then for any $\alpha\in(-\bar\beta,0]$ such that $\alpha+\bar\beta<\delta$ and $\alpha+\bar\beta\notin\bbN[\mfs]$, the map $f\mapsto Kf:=\int_0^1K_tfdt$ extends to a continuous linear operator from $B_p^{\alpha,Q}(w)$ to $B_p^{\alpha+\bar\beta}(w)$.
\end{lem}

\begin{proof}
By the density argument, it is sufficient to consider $f\in L_c^p(w)$.
By Definition \ref{asmp2}-\ref{asmp2:convolutionwithQ} and \ref{asmp2:gauss}, for any $|{\bf k}|_\mfs<\delta$ we have
\begin{align*}
\|\partial^{\bf k}K_tf\|_{L^p(w)}
&=\bigg\|\int_{\bbR^d}\partial_x^{\bf k}K_{t/2}(x,y)(Q_{t/2}f)(y)dy\bigg\|_{L_x^p(w)}\\
&\lesssim t^{(\bar\beta-|{\bf k}|_\mfs)/\ell-1}\bigg\|\int_{\bbR^d}G_{t/2}(x-y)|(Q_{t/2}f)(y)|dy\bigg\|_{L_x^p(w)}\\
&\lesssim t^{(\bar\beta-|{\bf k}|_\mfs)/\ell-1}\|Q_{t/2}f\|_{L^p(w)}
\lesssim t^{(\alpha+\bar\beta-|{\bf k}|_\mfs)/\ell-1}\|f\|_{B_p^{\alpha,Q}(w)}.
\end{align*}
This implies that the integral $\partial^{\bf k}Kf:=\int_0^1\partial^{\bf k}K_tfdt\in L_c^p(w)$ is well-defined for any $|{\bf k}|_\mfs<\alpha+\bar\beta$. To show the estimate of $T_{x;h}^{\alpha+\bar\beta-|{\bf k}|_\mfs}(\partial^{\bf k}Kf)$ for $|{\bf k}|_\mfs<\alpha+\bar\beta$, we divide the integral for $t$ into the regions $(0,t_0)$ and $[t_0,1)$ with $t_0:=\|h\|_\mfs^\ell\wedge1$.
In the region $(0,t_0)$, by Lemma \ref{lem:weightedtranslation} we have
\begin{align*}
&\int_0^{t_0}\|T_{x;h}^{\alpha+\bar\beta-|{\bf k}|_\mfs}(\partial^{\bf k}K_tf)\|_{L_x^p(w)}dt\\
&\lesssim \|f\|_{B_p^{\alpha,Q}(w)}\int_0^{t_0}\bigg\{w^*(h)t^{(\alpha+\bar\beta-|{\bf k}|_\mfs)/\ell-1}+\sum_{|{\bf l}|_\mfs<\alpha+\bar\beta-|{\bf k}|_\mfs}\|h\|_\mfs^{|{\bf l}|_\mfs}t^{(\alpha+\bar\beta-|{\bf k}|_\mfs-|{\bf l}|_\mfs)/\ell-1}\bigg\}dt\\
&\lesssim \|f\|_{B_p^{\alpha,Q}(w)}\bigg\{w^*(h)t_0^{(\alpha+\bar\beta-|{\bf k}|_\mfs)/\ell}+\sum_{|{\bf l}|_\mfs<\alpha+\bar\beta-|{\bf k}|_\mfs}\|h\|_\mfs^{|{\bf l}|_\mfs}t_0^{(\alpha+\bar\beta-|{\bf k}|_\mfs-|{\bf l}|_\mfs)/\ell}\bigg\}\\
&\lesssim \|f\|_{B_p^{\alpha,Q}(w)}w^*(h)\|h\|_\mfs^{\alpha+\bar\beta-|{\bf k}|_\mfs}.
\end{align*}
To consider the region $[t_0,1)$, we replace the condition ``$|{\bf l}|_\mfs<\alpha+\bar\beta-|{\bf k}|_\mfs$" in the sum for ${\bf l}$ with ``$|{\bf l}|_\mfs<\alpha+\gamma-|{\bf k}|_\mfs$" for some $\gamma\in(\bar\beta,\delta-\alpha)$. Such a choice is possible by assumption.
Since $\|h\|_\mfs=t_0^{1/\ell}\le t^{1/\ell}$ if $t\in [t_0,1)$, by Definition \ref{asmp2}-\ref{asmp2:convolutionwithQ} and \ref{asmp2:holder} we have
\begin{align*}
&\int_{t_0}^1\|T_{x;h}^{\alpha+\bar\beta-|{\bf k}|_\mfs}(\partial^{\bf k}K_tf)\|_{L_x^p(w)}dt\\
&=\int_{t_0}^1\bigg\|\int_{\bbR^d}\big\{T_{x;h}^{\alpha+\gamma-|{\bf k}|_\mfs}\big(\partial_{\cdot}^{\bf k}K_{t/2}(\cdot,y)\big)\big\}(Q_{t/2}f)(y)dy\bigg\|_{L_x^p(w)}dt\\
&\lesssim\int_{t_0}^1\bigg\|\int_{\bbR^d}\|h\|_\mfs^{\alpha+\gamma-|{\bf k}|_\mfs}t^{(\bar\beta-\gamma-\alpha)/\ell-1}G_{t/2}(x-y)|(Q_{t/2}f)(y)|dy\bigg\|_{L_x^p(w)}dt\\
&\lesssim \|h\|_\mfs^{\alpha+\gamma-|{\bf k}|_\mfs}\int_{t_0}^1t^{(\bar\beta-\gamma-\alpha)/\ell-1}\|Q_{t/2}f\|_{L^p(w)}dt\\
&\lesssim\|f\|_{B_p^{\alpha,Q}(w)}\|h\|_\mfs^{\alpha+\gamma-|{\bf k}|_\mfs}t_0^{(\bar\beta-\gamma)/\ell}
=\|f\|_{B_p^{\alpha,Q}(w)}\|h\|_\mfs^{\alpha+\bar\beta-|{\bf k}|_\mfs}.
\end{align*}
\end{proof}

\begin{rem}\label{rem:KQ=QK}
For the case $\alpha+\bar\beta<0$, we can show the continuity of $K:B_p^{\alpha,Q}(w)\to B_p^{\alpha+\bar\beta,Q}(w)$ if we assume the opposite convolution property
$$
\int_{\bbR^d}Q_{t-s}(x,z)K_s(z,y)dz=K_t(x,y)
$$
to Definition \ref{asmp2}-\ref{asmp2:convolutionwithQ}. This is the case for instance if $Q$ and $K$ are homogeneous; $Q_t(x,y)=Q_t(x-y)$ and $K_t(x,y)=K_t(x-y)$.
Indeed, since for any $f\in L_c^p(w)$,
\begin{align*}
\|Q_tKf\|_{L^p(w)}&=\bigg\|\int_0^1Q_tK_sfds\bigg\|_{L^p(w)}=\bigg\|\int_0^1K_{t+s}fds\bigg\|_{L^p(w)}\\
&=\bigg\|\int_0^1K_{(t+s)/2}Q_{(t+s)/2}fds\bigg\|_{L^p(w)}\\
&\lesssim\|f\|_{B_p^{\alpha,Q}(w)}\int_0^1(t+s)^{\bar\beta/\ell-1}(t+s)^{\alpha/\ell}ds\lesssim t^{(\alpha+\bar\beta)/\ell}\|f\|_{B_p^{\alpha,Q}(w)},
\end{align*}
we thus have $\|Kf\|_{B_p^{\alpha+\bar\beta,Q}(w)}\lesssim\|f\|_{B_p^{\alpha,Q}(w)}$.
\end{rem}

We prepare useful estimates for the proof of multilevel Schauder estimate.

\begin{lem}\label{lem:KregularizePi}
Let $w$ and $v$ be $G$-controlled weights such that $w^2$ and $wv$ are also $G$-controlled.
Let $\scT=(\bfA,\bfT,\bfG)$ be a regularity-integrability structure
and let $M=(\Pi,\Gamma)\in\scM_w(\scT)$.
For any $\bfa\in \bfA$, $\bfc\in\bbR\times[1,\infty]$ such that $\bfa\prec\bfc$, $|{\bf k}|_{\mfs}<\delta$, and $t\in(0,1]$, we have
$$
\big\|\partial^{\bf k}K_t\big(x,{\Pi}_x(\cdot)\big)\big\|_{L_x^{i(\bfa)}(w^2;\bfT_\bfa^*)}\lesssim C_K\|\Pi\|_{\bfc,w}(1+\|\Gamma\|_{\bfc,w})\,t^{(r(\bfa)+\bar\beta-|{\bf k}|_{\mfs})/\ell-1},
$$
where the implicit proportional constant depends only on $G,w$, and $\bfA$.
Consequently, if $|{\bf k}|_\mfs<(r(\bfa)+\bar\beta)\wedge\delta$, the operator
$$
\partial^{\bf k}K\big(x,\Pi_x(\cdot)\big):=\int_0^1\partial^{\bf k}K_t\big(x,\Pi_x(\cdot)\big)dt
$$
is well-defined in the class $L_x^{i(\bfa)}(w^2;\bfT_\bfa^*)$.
In addition, for any $f\in\mcD_{v}^\bfc(\Gamma)$ with $\bfc\in\bbR\times[1,\infty]$ and its reconstruction $\Lambda$, and any $|{\bf k}|_{\mfs}<\delta$ and $t\in(0,1]$, we have
$$
\big\|\partial^{\bf k}K_t(x,\Lambda)-\partial^{\bf k}K_t\big(x,{\Pi}_xf(x)\big)\big\|_{L_x^{i(\bfc)}(wv)}\lesssim C_K(\lb\Lambda\rb_{\bfc,wv}^{\Pi,f}+\|\Pi\|_{\bfc,w}\|f\|_{\bfc,v}^\Gamma)\,t^{(r(\bfc)+\bar\beta-|{\bf k}|_{\mfs})/\ell-1},
$$
where the implicit proportional constant depends only on $G,w,v$, and $\bfA$.
Consequently, the function $x\mapsto\partial^{\bf k}K(x,\Lambda)-\partial^{\bf k}K\big(x,{\Pi}_xf(x)\big)$ is well-defined as an element of $L_x^{i(\bfc)}(wv)$ if $|{\bf k}|_\mfs<(r(\bfc)+\bar\beta)\wedge\delta$.
\end{lem}

\begin{proof}
By Definition \ref{asmp2}-\ref{asmp2:convolutionwithQ} and \ref{asmp2:gauss}, for any $\tau\in\bfT_\bfa$,
\begin{align*}
|\partial^{\bf k}K_t(x,\Pi_x\tau)|
&=\bigg|\int_{\bbR^d}\partial^{\bf k}K_{t/2}(x,x-h)Q_{t/2}(x-h,\Pi_x\tau)dh\bigg|\\
&\lesssim C_Kt^{(\bar\beta-|{\bf k}|_{\mfs})/\ell-1}\int_{\bbR^d}G_{t/2}(h)|Q_{t/2}(x-h,\Pi_x\tau)|dh.
\end{align*}
By using the inequality \eqref{eq:Q(x-h,Pix)} obtained in Remark \ref{rem:whatisknownforZ}, we have
\begin{align*}
&\big\|\partial^{\bf k}K_t\big(x,{\Pi}_x(\cdot)\big)\big\|_{L_x^{i(\bfa)}(w^2;\bfT_\bfa^*)}\\
&\lesssim C_Kt^{(\bar\beta-|{\bf k}|_{\mfs})/\ell-1}\|\Pi\|_{\bfc,w}(1+\|\Gamma\|_{\bfc,w})\sum_{\bfb\preceq\bfa}t^{r(\bfb)/\ell}\int_{\bbR^d}G_{t/2}(h)\|h\|_{\mfs}^{r(\bfaob)}\big(w^*(h)\big)^2dh\\
&\lesssim C_Kt^{(r(\bfa)+\bar\beta-|{\bf k}|_{\mfs})/\ell-1}\|\Pi\|_{\bfc,w}(1+\|\Gamma\|_{\bfc,w}).
\end{align*}
For the remaining assertion, we have only to repeat the same argument as Remark \ref{rem:whatisknownforZ}, by replacing $\Pi_x\tau=\Pi_{x-h}\Gamma_{(x-h)x}\tau$ at the beginning of the proof with the identity
$$
\Lambda-\Pi_xf(x)=\big(\Lambda-\Pi_{x-h}f(x-h)\big)+\Pi_{x-h}\big(f(x-h)-\Gamma_{(x-h)x}f(x)\big).
$$
\end{proof}

\subsection{Abstract integrations and compatible models}

Throughout this section, we fix a $\bar\beta$-regularizing kernel $\{K_t\}_{t>0}$ admissible for $\{Q_t\}_{t>0}$.
In addition, we assume that there exists $G$-controlled weights $w_1$ and $w_2$ satisfying the assumption of Proposition \ref{prop:whatisPixtau} (only to ensure that $\Pi_x\tau$ is an element of some Besov space).
To lift the operator $K=\int_0^1K_tdt$ into the model space, we introduce the \emph{polynomial structure} generated by symbols $X_1,\dots,X_d$ as in \cite[Section 2]{Hai14}.
For any $\bfa\in\bbR\times[1,\infty]$ and $\beta>0$, we define the elements $\bfa\oplus\beta,\bfa\ominus\beta\in\bbR\times[1,\infty]$ by
\begin{align*}
\bfa\oplus\beta:=\big(r(\bfa)+\beta,i(\bfa)\big),\qquad
\bfa\ominus\beta:=\big(r(\bfa)-\beta,i(\bfa)\big).
\end{align*}

\begin{defi}\label{defTbarT}
Let $\bar\scT=(\bar{\bfA},\bar{\bfT},\bar{\bfG})$ be a regularity-integrability structure satisfying the following properties.
\begin{itemize}
\item[(1)]
$\bbN[\mfs]\times\{\infty\}\subset \bar{\bfA}$.
\item[(2)]
For each $\alpha\in\bbN[\mfs]$, the space $\bar{\bfT}_{(\alpha,\infty)}$ contains all $X^{\bf k}:=\prod_{i=1}^dX_i^{k_i}$ with $|{\bf k}|_{\mfs}=\alpha$.
\item[(3)]
The subspace $\spa\{X^{\bf k}\}_{{\bf k}\in\bbN^d}$ of $\bar\bfT$ is closed under $\bar{\bfG}$-actions.
\end{itemize}
Let $\scT=(\bfA,\bfT,\bfG)$ be another regularity-integrability structure.
A continuous linear operator $\mcI:\bfT\to\bar{\bfT}$ is called an \emph{abstract integration} of order $\beta\in(0,\bar\beta]$ if 
$$
\mcI:\bfT_\bfa\to\bar{\bfT}_{\bfa\oplus\beta}
$$
for any $\bfa\in \bfA$.
\end{defi}

\begin{defi}
Let $\scT$ and $\bar{\scT}$ be regularity-integrability structures as in Definition \ref{defTbarT}, and let $\mcI:\bfT\to\bar\bfT$ be an abstract integration of order $\beta\in(0,\bar\beta]$.
We say that the pair $(M,\bar{M})$ of two models $M=(\Pi,\Gamma)\in\scM_w(\scT)$ and $\bar{M}=(\bar{\Pi},\bar{\Gamma})\in\scM_w(\bar{\scT})$ with a $G$-controlled weight $w$ is \emph{compatible for $\mcI$} if it satisfies the following properties.
\begin{enumerate}
\renewcommand{\theenumi}{(\roman{enumi})}
\renewcommand{\labelenumi}{(\roman{enumi})}
\item\label{Iadmissible2}
For any ${\bf k}\in\bbN^d$,
$$
(\bar{\Pi}_xX^{\bf k})(\cdot)=(\cdot-x)^{\bf k},\qquad
\bar\Gamma_{yx} X^{\bf k}=\sum_{{\bf l}\le {\bf k}}\binom{\bf k}{\bf l}(y-x)^{\bf l} X^{{\bf k}-{\bf l}}.
$$
\item\label{Iadmissible3}
We define the linear map $\mcJ(x):\bfT_{\prec(\delta-\beta,1)}\to\spa\{X^{\bf k}\}_{|{\bf k}|_\mfs<\delta}\subset\bar{\bfT}$ by setting
\begin{align}\label{eq:Jx}
\mcJ(x)\tau=\sum_{|{\bf k}|_{\mfs}<r(\bfa)+\beta}\frac{X^{\bf k}}{{\bf k}!}\partial^{\bf k}K(x,\Pi_x\tau)
\end{align}
for any $\bfa\in \bfA$ such that $r(\bfa)+\beta<\delta$ and $\tau\in \bfT_{\bfa}$.
Then on the space $\bfT_{\prec(\delta-\beta,1)}$,
\begin{align}\label{ex:compatiblemodel}
\bar{\Gamma}_{yx}\big(\mcI+\mcJ(x)\big)\tau=\big(\mcI+\mcJ(y)\big)\Gamma_{yx}\tau.
\end{align}
\end{enumerate}
In addition, if the regularity $\alpha_0$ of $\scT$ is greater than $-\bar\beta$ and 
\begin{align}\label{ex:admissiblemodel}
(\bar{\Pi}_x\mcI\tau)(\cdot)=K(\cdot,\Pi_x\tau)-\sum_{|{\bf k}|_{\mfs}<r(\bfa)+\beta}\frac{(\cdot-x)^{\bf k}}{{\bf k}!}\partial^{\bf k}K(x,\Pi_x\tau),
\end{align}
(recall from Proposition \ref{prop:whatisPixtau} and Lemma \ref{lem:whyregularize} that the right-hand side is well-defined)
for any $\tau\in\bfT_\bfa$ with $r(\bfa)+\beta<\delta$, then we say that the pair $(M,\bar{M})$ is \emph{$K$-admissible for $\mcI$}.
\end{defi}

\begin{rem}
The above definition is a modification of the original one \cite{Hai14}.
Indeed, the abstract integration is defined between distinct regularity-integrability structures.
Moreover, the condition \eqref{ex:compatiblemodel} is separated from the $K$-admissibility of the model, while \eqref{ex:compatiblemodel} was a result of \eqref{ex:admissiblemodel} in \cite[Lemma 5.16]{Hai14}.
In the paper \cite{BPHZ}, we consider the situation where only \eqref{ex:compatiblemodel} holds.
\end{rem}

\begin{rem}
The quantity \eqref{eq:Jx} is only defined for almost every $x\in\bbR^d$ for $\bfa$ such that $i(\bfa)<\infty$, since elements of $L_c^{i(\bfa)}(w)$ may not be continuous. Hence there is a subtle problem that the negligible set may depend on $\tau$.
However, since we can define $\partial^{\bf k}K\big(x,\Pi_x(\cdot)\big)$ as a $\bfT_\bfa^*$-valued function of class $L^p(w^2)$, the negligible set can be chosen $\tau$-independently. Similarly, we understand \eqref{ex:compatiblemodel} and \eqref{ex:admissiblemodel} as identities for operators of $\tau$ which hold for almost every $x,y\in\bbR^d$.
\end{rem}

\subsection{Multilevel Schauder estimate in regularity-integrability structures}

In what follows, we fix regularity-integrability structures $\scT$ and $\bar{\scT}$ satisfying the setting of Definition \ref{defTbarT} and an abstract integration $\mcI$.
Moreover, let $w$ and $v$ be $G$-controlled weights such that $w^2v$ is also $G$-controlled.

\begin{defi}
For any $(\Pi,\Gamma)\in\scM_{w}(\scT)$, $f\in\mcD_v^\bfc(\Gamma)$ with $\bfc\in\bbR\times[1,\infty]$ such that $r(\bfc)+\beta<\delta$, and its reconstruction $\Lambda$, we define
$$
\mcN(x;f,\Lambda)=\sum_{|{\bf k}|_{\mfs}<r(\bfc)+\beta}\frac{X^{\bf k}}{{\bf k}!}\partial^{\bf k}K\big(x,\Lambda-\Pi_xf(x)\big)
$$
and
$$
\mcK f(x):=\mcI f(x)+\mcJ(x)f(x)+\mcN(x;f,\Lambda).
$$
\end{defi}

\begin{thm}\label{thm:besovschauder}
Let $\bfc\in(-\infty,\delta-\beta)\times[1,\infty]$ and assume either of the following conditions.
\begin{itemize}
\item[(1)]
$\beta<\bar\beta$.
\item[(2)]
$\beta=\bar\beta$, $\{r(\bfa)+\bar\beta\, ;\, \bfa\in\bfA\}\cap\bbN[\mfs]=\emptyset$, and $r(\bfc)+\bar\beta\notin\bbN[\mfs]$.
\end{itemize}
Then for any compatible pair of models $\big(M=(\Pi,\Gamma),\bar{M}=(\bar{\Pi},\bar{\Gamma})\big)\in\scM_w(\scT)\times\scM_w(\bar{\scT})$,
modelled distribution $f\in\mcD_{v}^\bfc(\Gamma)$, and any reconstruction $\Lambda$ of $f$ for $M$, the function $\mcK f$ belongs to $\mcD_{w^2v}^{\bfcpb}(\bar\Gamma)$, and we have
\begin{align}
\label{MSbesov1}
\lp \mcK f\rp_{\bfcpb,w^2v}&\lesssim
\|\mcI\|\lp f\rp_{\bfc,v}+C_K\big\{\|\Pi\|_{\bfc,w}(1+\|\Gamma\|_{\bfc,w})\tri f\tri_{\bfc,v}^\Gamma+\lb\Lambda\rb_{\bfc,wv}^{\Pi,f}\big\},\\
\label{MSbesov2}
\|\mcK f\|_{\bfcpb,w^2v}^{\bar\Gamma}
&\lesssim\|\mcI\|\| f\|_{\bfc,v}^\Gamma+C_K\big\{\|\Pi\|_{\bfc,w}(1+\|\Gamma\|_{\bfc,w})\| f\|_{\bfc,v}^\Gamma+\lb\Lambda\rb_{\bfc,wv}^{\Pi,f}\big\},
\end{align}
where $\|\mcI\|$ is the operator norm from $\bfT_{\prec\bfc}$ to $\bar{\bfT}_{\prec\bfcpb}$, and the implicit proportional constant depends only on $G,w,v,\bfc$, and $\bfA$.
Moreover, there is a quadratic function $C_R>0$ of $R>0$ such that
\begin{align*}
\tri\mcK f^{(1)};\mcK f^{(2)}\tri_{\bfcpb,w^2v}
\le C_R\big(
\tri M^{(1)};M^{(2)}\tri_{\bfc,w}+\tri f^{(1)};f^{(2)}\tri_{\bfc,v}
+\lb\Lambda^{(1)};\Lambda^{(2)}\rb_{\bfc,wv}
\big)
\end{align*}
for any $M^{(i)}=(\Pi^{(i)},\Gamma^{(i)})\in\scM_{w}(\scT)$ and $\bar{M}^{(i)}=(\bar\Pi^{(i)},\bar\Gamma^{(i)})\in\scM_w(\bar\scT)$ such that $(M^{(i)},\bar{M}^{(i)})$ is compatible, any $f^{(i)}\in\mcD_{v}^\bfc(\Gamma^{(i)})$, and any reconstructions $\Lambda^{(i)}$ of $f^{(i)}$ for $M^{(i)}$ with $i\in\{1,2\}$ such that $\tri M^{(i)}\tri_{\bfc,w}\le R$ and $\tri f^{(i)}\tri_{\bfc,v}^{\Gamma^{(i)}}\le R$.
\end{thm}

\begin{proof}
The proof of the local Lipschitz estimate is a simple modification of those of \eqref{MSbesov1} and \eqref{MSbesov2}. The bound \eqref{MSbesov1} immediately follows from the continuity of $\mcI$ and Lemma \ref{lem:KregularizePi}. In the following, we focus on the proof of \eqref{MSbesov2}.
By using the property \eqref{ex:compatiblemodel} of compatible models, we decompose
\begin{align*}
\Delta_{x;h}^{\bar\Gamma}\mcK f&=\mcK f(x-h)-\bar\Gamma_{(x-h)x}\mcK f(x)\\
&=\mcK f(x-h)-\bar\Gamma_{(x-h)x}\big(\mcI+\mcJ(x)\big)f(x)-\bar\Gamma_{(x-h)x}\,\mcN(x;f,\Lambda)\\
&=\mcK f(x-h)-\big(\mcI+\mcJ(x-h)\big)\Gamma_{(x-h)x}f(x)-\bar\Gamma_{(x-h)x}\,\mcN(x;f,\Lambda)\\
&=\big(\mcI+\mcJ(x-h)\big)(\Delta_{x;h}^\Gamma f)
+\big(\mcN(x-h;f,\Lambda)-\bar\Gamma_{(x-h)x}\,\mcN(x;f,\Lambda)\big)\\
&=:\mcI(\Delta_{x;h}^\Gamma f)+\sum_{|{\bf k}|_{\mfs}<r(\bfc)+\beta}\frac{X^{\bf k}}{{\bf k}!}\mcA^{\bf k}(x;h).
\end{align*}
For the $\mcI$ term, noting that $(\bfcpb)\ominus\bfa=\bfc\ominus(\bfa\ominus\beta)$ we easily obtain
\begin{align*}
\|\mcI(\Delta_{x;h}^\Gamma f)\|_{L_x^{i((\bfcpb)\ominus\bfa)}(v;\bfT_\bfa)}
&\le\|\mcI\|\|\Delta_{x;h}^\Gamma f\|_{L_x^{i(\bfc\ominus(\bfa\ominus\beta))}(v;\bfT_{\bfa\ominus\beta})}\\
&\le\|\mcI\|\|f\|_{\bfc,v}^\Gamma\, v^*(h)\|h\|_\mfs^{r((\bfcpb)\ominus\bfa)}.
\end{align*}
When $r(\bfc)+\beta\le0$ the proof is completed.
In the rest of the proof, we assume $r(\bfc)+\beta>0$ and focus on the polynomial part.
Since $X^{{\bf k}}$ belongs to the space $\bar{\bf T}_{(|{\bf k}|_{\mfs},\infty)}$ and $i((\bfc\oplus\beta)\ominus(|{\bf k}|_\mfs,\infty))=i(\bfc)$, we check the $L_x^{i(\bfc)}$ norm of $\mcA^{\bf k}(x;h)$.
Note that the coefficient $\mcA^{\bf k}$ is given by
\begin{align*}
\mcA^{\bf k}(x;h)&=\sum_{\bfa\in \bfA,\,r(\bfa)>|{\bf k}|_{\mfs}-\beta}\partial^{\bf k}K\big(x-h,\Pi_{x-h} P_\bfa\Delta_{x;h}^\Gamma f\big)\\
&\quad+\bigg\{\partial^{\bf k}K\big(x-h, \Lambda_{x-h}^{\Pi,f}\big)
-\sum_{|{\bf l}|_\mfs<r(\bfc)+\beta-|{\bf k}|_\mfs}\frac{(-h)^{\bf l}}{{\bf l}!}\partial^{{\bf k}+{\bf l}}K\big(x,\Lambda_x^{\Pi,f}\big)\bigg\},
\end{align*}
where $\Lambda_x^{\Pi.f}:=\Lambda-\Pi_xf(x)$.
According to the integral form $K=\int_0^1K_tdt$, we decompose $\mcA^{\bf k}=\int_0^1\mcA_t^{\bf k}dt=\int_0^1(\mcB_t^{{\bf k},1}+\mcB_t^{{\bf k},2})dt$, where
\begin{align*}
\mcB_t^{{\bf k},1}(x;h)&:=\sum_{\bfa\in \bfA,\,r(\bfa)>|{\bf k}|_{\mfs}-\beta}\partial^{\bf k}K_t\big(x-h,\Pi_{x-h} P_\bfa\Delta_{x;h}^\Gamma f\big),\\
\mcB_t^{{\bf k},2}(x;h)&:=\partial^{\bf k}K_t\big(x-h, \Lambda_{x-h}^{\Pi,f}\big)
-\sum_{|{\bf l}|_\mfs<r(\bfc)+\beta-|{\bf k}|_\mfs}\frac{(-h)^{\bf l}}{{\bf l}!}\partial^{{\bf k}+{\bf l}}K_t\big(x,\Lambda_x^{\Pi,f}\big).
\end{align*}
We use this decomposition for the integral over $0\le t\le t_0:=\|h\|_\mfs^\ell\wedge1$.
For the $\mcB_t^{{\bf k},1}$ part, by Lemmas \ref{lem:weightedtranslation} and \ref{lem:KregularizePi},
\begin{align*}
\|\mcB_t^{{\bf k},1}(x;h)\|_{L_x^{i(\bfc)}(w^2v)}
&\le\sum_{r(\bfa)>|{\bf k}|_{\mfs}-\beta}\|\Delta_{x;h}^\Gamma f\|_{L_x^{i(\bfcoa)}(v;\bfT_\bfa)}
\big\|\partial^{\bf k}K_t\big(x-h,\Pi_{x-h}(\cdot)\big)\big\|_{L_x^{i(\bfa)}(w^2;\bfT_\bfa^*)}\\
&\lesssim L_1\big((\wco)^2\vco\big)(h)
\sum_{r(\bfa)>|{\bf k}|_{\mfs}-\beta}\|h\|_{\mfs}^{r(\bfcoa)}t^{(r(\bfa)+\bar\beta-|{\bf k}|_{\mfs})/\ell-1}\\
&\le L_1\big((\wco)^2\vco\big)(h)
\sum_{r(\bfa)>|{\bf k}|_{\mfs}-\beta}\|h\|_{\mfs}^{r(\bfcoa)}t^{(r(\bfa)+\beta-|{\bf k}|_{\mfs})/\ell-1},
\end{align*}
where $L_1=C_K\|\Pi\|_{\bfc,w}(1+\|\Gamma\|_{\bfc,w})\|f\|_{\bfc,v}^\Gamma$.
In the last inequality, we used $t\le1$ and $\beta\le\bar\beta$.
For the $\mcB_t^{{\bf k},2}$ part, by Lemmas \ref{lem:weightedtranslation} and \ref{lem:KregularizePi},
\begin{align*}
\|\mcB_t^{{\bf k},2}(x;h)\|_{L_x^{i(\bfc)}(wv)}
&\lesssim L_2\,t^{(r(\bfc)+\bar\beta-|{\bf k}|_\mfs)/\ell-1}
\bigg((\wco\vco)(h)+\sum_{|{\bf l}|_\mfs<r(\bfc)+\beta-|{\bf k}|_\mfs}\|h\|_\mfs^{|{\bf l}|_\mfs}t^{-|{\bf l}|_\mfs/\ell}\bigg)\\
&\le L_2\,t^{(r(\bfc)+\beta-|{\bf k}|_\mfs)/\ell-1}
\bigg((\wco\vco)(h)+\sum_{|{\bf l}|_\mfs<r(\bfc)+\beta-|{\bf k}|_\mfs}\|h\|_\mfs^{|{\bf l}|_\mfs}t^{-|{\bf l}|_\mfs/\ell}\bigg),
\end{align*}
where $L_2=C_K(\lb\Lambda\rb_{\bfc,wv}^{\Pi,f}+\|\Pi\|_{\bfc,w}\|f\|_{\bfc,v})$.
Since all powers of $t$ above are greater than $-1$, we have the bound
\begin{align*}
\int_0^{t_0}\|\mcA_t^{\bf k}(x;h)\|_{L_x^{i(\bfc)}(w^2v)}dt
&\lesssim(L_1+L_2)\big((\wco)^2\vco\big)(h)
\sum_{\alpha_1+\alpha_2=r(\bfc)+\beta-|{\bf k}|_{\mfs}}\|h\|_{\mfs}^{\alpha_1}\, t_0^{\alpha_2/\ell}\\
&\lesssim(L_1+L_2)\big((\wco)^2\vco\big)(h)
\|h\|_{\mfs}^{r(\bfc)+\beta-|{\bf k}|_{\mfs}}.
\end{align*}
Finally, we assume that $\|h\|_\mfs\le1$ and consider the integral over $t_0=\|h\|_\mfs^\ell\le t\le1$.
For this case, we use another decomposition $\mcA_t^{\bf k}=\mcC_t^{{\bf k},1}+\mcC_t^{{\bf k},2}$ given by
\begin{align*}
\mcC_t^{{\bf k},1}(x;h)&:=-\sum_{\bfa\in \bfA,\,r(\bfa)\le|{\bf k}|_{\mfs}-\beta}\partial^{\bf k}K_t\big(x-h,\Pi_{x-h} P_\bfa\Delta_{x;h}^\Gamma f\big),\\
\mcC_t^{{\bf k},2}(x;h)&:=\partial^{\bf k}K_t\big(x-h, \Lambda_x^{\Pi,f}\big)
-\sum_{|{\bf l}|_\mfs<r(\bfc)+\beta-|{\bf k}|_\mfs}\frac{(-h)^{\bf l}}{{\bf l}!}\partial^{{\bf k}+{\bf l}}K_t\big(x,\Lambda_x^{\Pi,f}\big),
\end{align*}
where we used the identity $\Lambda_{x-h}^{\Pi,f}=\Lambda_x^{\Pi,f}-\Pi_{x-h}\Delta_{x;h}^\Gamma f$ to show $\mcB_t^{{\bf k},1}+\mcB_t^{{\bf k},2}=\mcC_t^{{\bf k},1}+\mcC_t^{{\bf k},2}$.
The bound of $\mcC_t^{{\bf k},1}$ is obtained similarly to $\mcB_t^{{\bf k},1}$ as follows.
\begin{align*}
&\|\mcC_t^{{\bf k},1}(x;h)\|_{L_x^{i(\bfc)}(w^2v)}
\lesssim L_1\big((\wco)^2\vco\big)(h)
\sum_{r(\bfa)\le |{\bf k}|_{\mfs}-\beta}\|h\|_{\mfs}^{r(\bfcoa)}t^{(r(\bfa)+\bar\beta-|{\bf k}|_{\mfs})/\ell-1}.
\end{align*}
However, for the integral $\int_{t_0}^1\mcC_t^{{\bf k},1}dt$, we have to pay more attention to the powers of $t$. For $\bfa$ such that $r(\bfa)<|{\bf k}|_\mfs-\beta$, we easily have
\begin{align*}
\int_{t_0}^1\|h\|_{\mfs}^{r(\bfcoa)}t^{(r(\bfa)+\bar\beta-|{\bf k}|_{\mfs})/\ell-1}dt
&\le\int_{t_0}^1\|h\|_{\mfs}^{r(\bfcoa)}t^{(r(\bfa)+\beta-|{\bf k}|_{\mfs})/\ell-1}dt\\
&\lesssim\|h\|_{\mfs}^{r(\bfcoa)}t_0^{(r(\bfa)+\beta-|{\bf k}|_\mfs)/\ell}
=\|h\|_\mfs^{r(\bfc)+\beta-|{\bf k}|_\mfs}.
\end{align*}
If there is $\bfa$ such that $r(\bfa)=|{\bf k}|_\mfs-\beta$, then since $\bar\beta>\beta$ by assumption, we have
\begin{align*}
\int_{t_0}^1\|h\|_{\mfs}^{r(\bfcoa)}t^{(r(\bfa)+\bar\beta-|{\bf k}|_{\mfs})/\ell-1}dt
=\int_{t_0}^1\|h\|_{\mfs}^{r(\bfc)+\beta-|{\bf k}|_\mfs}t^{(\bar\beta-\beta)/\ell-1}dt
\lesssim\|h\|_\mfs^{r(\bfc)+\beta-|{\bf k}|_\mfs}.
\end{align*}
For the $\mcC_t^{{\bf k},2}$ part, we employ the inequality obtained in Remark \ref{asmp2'} with $\varepsilon=r(\bfc)+\beta$ (recall that we consider the case $r(\bfc)+\beta>0$) and have
\begin{align}\label{ex:Ck2}
\begin{aligned}
&|\mcC_t^{{\bf k},2}(x;h)|\\
&=\bigg|\int_{\bbR^d}\bigg(\partial_y^{\bf k}K_{t/2}(x-h,y)-\sum_{|{\bf l}|_\mfs<r(\bfc)+\beta-|{\bf k}|_\mfs}\frac{(-h)^{\bf l}}{{\bf l}!}\partial_x^{{\bf k}+{\bf l}}K_{t/2}(x,y)\bigg)
Q_{t/2}\big(y,\Lambda_x^{\Pi,f}\big)dy\bigg|\\
&\lesssim C_K\|h\|_\mfs^{r(\bfc)+\beta-|{\bf k}|_\mfs}t^{(\bar\beta-\beta-r(\bfc))/\ell-1}\int_{\bbR^d}G_{t/2}(x-y)\big|Q_{t/2}(y,\Lambda_x^{\Pi,f})\big|dy.
\end{aligned}
\end{align}
By taking $L_x^{i(\bfc)}(wv)$ norm, we have
\begin{align*}
&\|\mcC_t^{{\bf k},2}(x;h)\|_{L_x^{i(\bfc)}(wv)}\\
&\lesssim C_K\|h\|_\mfs^{r(\bfc)+\beta-|{\bf k}|_\mfs}t^{(\bar\beta-\beta-r(\bfc))/\ell-1}
\int_{\bbR^d}G_{t/2}(z)\big\|Q_{t/2}(x-z,\Lambda_x^{\Pi,f})\big\|_{L_x^{i(\bfc)}(wv)}dz\\
&\lesssim L_2\|h\|_\mfs^{r(\bfc)+\beta-|{\bf k}|_\mfs}t^{(\bar\beta-\beta-r(\bfc))/\ell-1}\int_{\bbR^d}G_{t/2}(z)(\wco\vco)(z)\bigg(t^{r(\bfc)/\ell}+\sum_{\bfa\prec \bfc}t^{r(\bfa)/\ell}\|z\|_{\mfs}^{r(\bfcoa)}\bigg)dz\\
&\lesssim L_2\|h\|_\mfs^{r(\bfc)+\beta-|{\bf k}|_\mfs}t^{(\bar\beta-\beta)/\ell-1},
\end{align*}
where the second inequality follows from a similar argument to \eqref{eq:Q(x-h,Pix)} in Remark \ref{rem:whatisknownforZ}.
For the case $\beta<\bar\beta$, we have the result by
\begin{align*}
\int_{t_0}^1\|h\|_{\mfs}^{r(\bfc)+\beta-|{\bf k}|_\mfs}t^{(\bar\beta-\beta)/\ell-1}dt
&\lesssim\|h\|_{\mfs}^{r(\bfc)+\beta-|{\bf k}|_\mfs}.
\end{align*}
If $\beta=\bar\beta$, we return to \eqref{ex:Ck2} and replace the region ``$|{\bf l}|_\mfs<r(\bfc)+\bar\beta-|{\bf k}|_\mfs$" with ``$|{\bf l}|_\mfs<r(\bfc)+\gamma-|{\bf k}|_\mfs$" for some $\gamma>\bar\beta$. This is possible because $r(\bfc)+\bar\beta\notin\bbN[\mfs]$ and $r(\bfc)+\bar\beta<\delta$ by assumption. Then by repeating the same argument as above, we have
\begin{align*}
\int_{t_0}^1\|\mcC_t^{{\bf k},2}(x;h)\|_{L_x^{i(\bfc)}(wv)}dt
&\lesssim L_2\|h\|_{\mfs}^{r(\bfc)+\gamma-|k|_{\mfs}}\int_{t_0}^1t^{(\bar\beta-\gamma)/\ell-1}dt\\
&\lesssim L_2\|h\|_{\mfs}^{r(\bfc)+\gamma-|{\bf k}|_\mfs}t_0^{(\bar\beta-\gamma)/\ell}
=L_2\|h\|_\mfs^{r(\bfc)+\bar\beta-|{\bf k}|_\mfs}.
\end{align*}
\end{proof}

We close this section with the important commutation result.

\begin{thm}\label{thm:KR=RK}
In addition to the setting of Theorem \ref{thm:besovschauder}, we assume that $\alpha_0+\bar\beta\in(0,\delta)\setminus\bbN[\mfs]$ 
for the regularity $\alpha_0$ of $\scT$ and that $(M,\bar{M})$ is $K$-admissible for $\mcI$. Then $K\Lambda\in B_{i(\bfc)}^{\alpha_0+\bar\beta}(w)$ is the unique reconstruction of $\mcK f\in\mcD_{w^2v}^{\bfcpb}(\bar\Gamma)$ and
$$
\lb K\Lambda\rb_{\bfcpb,w^2v}^{\bar{\Pi},\mcK f}\lesssim C_K\big(\lb\Lambda\rb_{\bfc,wv}^{\Pi,f}+\|\Pi\|_{\bfc,w}\|f\|_{\bfc,v}^\Gamma\big).
$$
Moreover, a similar local Lipschitz estimate to the latter part of Theorem \ref{thm:besovschauder} holds.
\end{thm}

\begin{rem}
The condition on $\alpha_0$ is only to ensure the existence of $K\Lambda$ as an element of Besov space. 
If $K\Lambda$ is well-defined even though $\alpha_0+\bar\beta\le0$ (cf. Remark \ref{rem:KQ=QK}), the same result as above holds under the weaker condition that $r(\bfc)+\beta>0$ which ensures the uniqueness of the reconstruction of $\mcK f$.
\end{rem}

\begin{proof}
By definition, we can write
\begin{align*}
\bar{\Pi}_x\mcK f(x)&=\bar{\Pi}_x\big(\mcI+\mcJ(x)\big)f(x)+\bar{\Pi}_x\,\mcN(x;f,\Lambda)\\
&=K\big(\cdot,\Pi_xf(x)\big)+\sum_{|{\bf k}|_\mfs<r(\bfc)+\beta}\frac{(\cdot-x)^{\bf k}}{{\bf k}!}\partial^{\bf k}K\big(x,\Lambda_x^{\Pi,f}\big)
\end{align*}
and thus
\begin{align*}
(K\Lambda)_x^{\bar{\Pi},\mcK f}=K\big(\cdot,\Lambda_x^{\Pi,f}\big)-\sum_{|{\bf k}|_\mfs<r(\bfc)+\beta}\frac{(\cdot-x)^{\bf k}}{{\bf k}!}\partial^{\bf k}K\big(x,\Lambda_x^{\Pi,f}\big)
=\int_0^1\mcC_s^{{\bf 0},2}(x;x-\cdot)ds
\end{align*}
by using the notation introduced in the proof of Theorem \ref{thm:besovschauder}.
By the bound of $\mcC^{{\bf k},2}$ obtained there, when $\beta<\bar\beta$ we have
\begin{align*}
\big\|Q_t\big(x,(K\Lambda)_x^{\bar{\Pi},\mcK f}\big)\big\|_{L_x^{i(\bfc)}(w^2v)}
&\le\bigg\|\int_0^1ds\int_{\bbR^d}Q_t(x,x-h)\mcC_s^{{\bf 0},2}(x;h)dh\bigg\|_{L_x^{i(\bfc)}(wv)}\\
&\lesssim\int_0^1ds\int_{\bbR^d}G_t(h)\|\mcC_s^{{\bf 0},2}(x;h)\|_{L_x^{i(\bfc)}(wv)}dh\\
&\lesssim L_2
\int_0^1ds\int_{\bbR^d}G_t(h)\|h\|_\mfs^{r(\bfc)+\beta}s^{(\bar\beta-\beta)/\ell-1}ds\\
&\lesssim L_2\, t^{(r(\bfc)+\beta)/\ell}.
\end{align*}
The proof for the case $\beta=\bar\beta$ is similar.
\end{proof}

\vspace{7mm}
\noindent
{\bf Acknowledgements.}
The author is supported by JSPS KAKENHI Grant Number 23K12987.
He also thank anonymous referees for their helpful comments that improved the
quality of the paper.

\end{document}